\documentclass[12pt]{amsart}

\usepackage{mathpazo}

\theoremstyle{plain}
\newtheorem{theorem}{Theorem}[section]
\newtheorem{proposition}[theorem]{Proposition}
\newtheorem{lemma}[theorem]{Lemma}

\newtheorem{fact}[theorem]{Fact}

\theoremstyle{definition}
\newtheorem{definition}[theorem]{Definition}

\theoremstyle{remark}
\newtheorem{remark}[theorem]{Remark}

\makeatletter
\newcommand{\opnorm}{\@ifstar\@opnorms\@opnorm}
\newcommand{\@opnorms}[1]{%
  \left|\mkern-1.5mu\left|\mkern-1.5mu\left|
   #1
  \right|\mkern-1.5mu\right|\mkern-1.5mu\right|
}
\newcommand{\@opnorm}[2][]{%
  \mathopen{#1|\mkern-1.5mu#1|\mkern-1.5mu#1|}
  #2
  \mathclose{#1|\mkern-1.5mu#1|\mkern-1.5mu#1|}
}
\makeatother

\begin{document}

\title[Amalgamations of Banach spaces]{Amalgamations of classes of Banach spaces with~a~monotone basis}
\author{Ond\v{r}ej Kurka}
\thanks{The research was supported by the grant GA\v{C}R 14-04892P. The author is a junior researcher in the University Centre for Mathematical Modelling, Applied Analysis and Computational Mathematics (MathMAC). The author is a member of the Ne\v{c}as Center for Mathematical Modeling.}
\address{Department of Mathematical Analysis, Charles University, Soko\-lovsk\'a 83, 186 75 Prague 8, Czech Republic}
\email{kurka.ondrej@seznam.cz}
\begin{abstract}
It was proved by Argyros and Dodos that, for many classes $ \mathcal{C} $ of separable Banach spaces which share some property $ P $, there exists an isomorphically universal space that satisfies $ P $ as well. We introduce a variant of their amalgamation technique which provides an isometrically universal space in the case that $ \mathcal{C} $ consists of spaces with a monotone Schauder basis. For example, we prove that if $ \mathcal{C} $ is a set of separable Banach spaces which is analytic with respect to the Effros-Borel structure and every $ X \in \mathcal{C} $ is reflexive and has a monotone Schauder basis, then there exists a separable reflexive Banach space that is isometrically universal for $ \mathcal{C} $.
\end{abstract}
\maketitle

\section{Introduction and the main result}

Let $ \mathcal{C} $ be a class of Banach spaces. We say that a Banach space $ X $ is \emph{isomorphically (isometrically) universal for $ \mathcal{C} $} if it contains an isomorphic (isometric) copy of every member of $ \mathcal{C} $.

The present paper deals with universality questions in separable Banach space theory. Our aim is to find an isometric version of the amalgamation theory of S.~A.~Argyros and P.~Dodos \cite{argyrosdodos} and provide a method how to construct small isometrically universal spaces for small families of Banach spaces. Many of the results considered in the paper employs methods from descriptive set theory. The connection of universality problems and descriptive set theory, discovered by J.~Bourgain \cite{bourgain1, bourgain2}, deepened the theory and enabled several intrinsic questions to be understood. (See also \cite{bossard2}, \cite{dodosferenczi}, \cite{dodos}, \cite{feloro}, for an introduction, see \cite{godefroypersp}).

In 1968, W.~Szlenk \cite{szlenk} proved that the class of separable reflexive spaces has no isomorphically universal element. (It had been shown some time ago by J.~Lindenstrauss \cite{lindenstrauss} that it has no isometrically universal element). He proved that a Banach space which is isomorphically universal for separable reflexive spaces has non-separable dual. His proof led to the famous Szlenk index which will be useful also in proofs of the present results.

Later, J.~Bourgain \cite{bourgain1} proved that, if a separable Banach space is isomorphically universal for separable reflexive spaces, then it is actually isomorphically universal for all separable Banach spaces. A somewhat different proof of this result was provided by B.~Bossard \cite{bossard2} who showed that, if an analytic set of separable Banach spaces contains all separable reflexive spaces up to isomorphism, then it contains a space which is isomorphically universal for all separable Banach spaces. (An analytic set of Banach spaces is defined in Section~\ref{sec:prelim}). For a separable Banach space $ X $, the set of all Banach spaces with an isomorphic copy in $ X $ is analytic. Therefore, Bourgain's result follows from Bossard's one.

Bossard's approach consists in constructing a tree space such that every infinite branch supports a universal space and every tree without infinite branches supports a reflexive space. It is possible to apply this approach on analogous questions concerning isometry as well. It was shown in \cite{godkal} that, if a separable Banach space is isometrically universal for separable strictly convex spaces, then it is actually isometrically universal for all separable Banach spaces. The same result holds for the class of reflexive spaces \cite{kurka}.

In a work of S.~A.~Argyros and P.~Dodos \cite{argyrosdodos}, the concept of a tree space turned out to be a powerful tool also for constructing universal spaces (see also \cite{dodostopics}). When a set of separable Banach spaces $ \mathcal{C} $ is simple (in the sense that $ \mathcal{C} $ is analytic and every member has a Schauder basis), then it is possible to find a tree space such that the spaces supported by infinite branches are isomorphic copies of all members of $ \mathcal{C} $. If the tree space is constructed properly, properties of spaces from $ \mathcal{C} $ can be preserved.

Some results of the Argyros-Dodos amalgamation theory are resumed in the following theorem (by a basis we mean a Schauder basis).

\begin{theorem}[\cite{argyrosdodos}] \label{thmargdod}
Let $ \mathcal{P} $ be one of the following classes of separable Banach spaces:
\begin{itemize}
\item the class of spaces with a shrinking basis,
\item the class of reflexive spaces with a basis,
\item the class of spaces with a basis which are not isomorphically universal for all separable Banach spaces.
\end{itemize}
Let $ \mathcal{C} $ be an analytic set of spaces from $ \mathcal{P} $. Then there exists a Banach space $ E $ which belongs to $ \mathcal{P} $ and which contains a complemented isomorphic copy of every member of $ \mathcal{C} $.
\end{theorem}

The reliance on a basis was dropped soon in works of P.~Dodos and V.~Ferenczi \cite{dodosferenczi} and P.~Dodos \cite{dodos}. They proved that Theorem~\ref{thmargdod} holds (without the property that the copies are complemented) also for the following classes:
\begin{itemize}
\item the class of spaces with separable dual \cite{dodosferenczi},
\item the class of separable reflexive spaces \cite{dodosferenczi},
\item the class of separable spaces which are not isomorphically universal for all separable Banach spaces \cite{dodos}.
\end{itemize}

In the present work, we study the problem whether these results have an isometric version (see also \cite[Problem 9]{godefroyprobl}). We establish an isometric variant of Theorem~\ref{thmargdod}.

A basis $ x_{1}, x_{2}, \dots $ is said to be \emph{monotone} if the associated partial sum operators $ P_{n} : \sum_{k=1}^{\infty} a_{k}x_{k} \mapsto \sum_{k=1}^{n} a_{k}x_{k} $ satisfy $ \Vert P_{n} \Vert \leq 1 $.

\begin{theorem} \label{thmmain}
Let $ \mathcal{P} $ be one of the following classes of separable Banach spaces:
\begin{itemize}
\item the class of spaces with a monotone shrinking basis,
\item the class of reflexive spaces with a monotone basis,
\item the class of spaces with a monotone basis which are not isometrically universal for all separable Banach spaces,
\item the class of strictly convex spaces with a monotone basis.
\end{itemize}
Let $ \mathcal{C} $ be an analytic set of spaces from $ \mathcal{P} $. Then there exists a Banach space $ E $ which belongs to $ \mathcal{P} $ and which contains an $ 1 $-complemented isometric copy of every member of $ \mathcal{C} $.
\end{theorem}

We do not know whether the reliance on a basis can be dropped, similarly as in the isomorphic setting. The requirement of the existence of a monotone basis is a weak point of Theorem~\ref{thmmain}, but it is possible that the theorem will be helpful for more powerful results to be obtained in future.

We include here several remarks concerning Theorem~\ref{thmmain}.

(I) For the class of spaces with a shrinking basis and the class of reflexive spaces with a basis, it is not difficult to show that Theorem~\ref{thmargdod} follows from Theorem~\ref{thmmain}.

(II) The theorem remains valid if we consider monotone finite dimensional decompositions instead of monotone bases. A variant of a space constructed by S.~Prus \cite{prus} can be provided. Since the class of super-reflexive spaces is analytic, there exists a separable reflexive space which contains an $ 1 $-complemented isometric copy of every super-reflexive space with a monotone finite dimensional decomposition. Since $ F \oplus_{2} \ell_{2} $ is super-reflexive for each finite dimensional space $ F $, we obtain also the result of A.~Szankowski \cite{szankowski} which states that there exists a separable reflexive space, isometrically universal for all finite dimensional spaces.

(III) It is possible to use the methods developed in the paper for the construction of a Pe\l czy\'nski universal space which contains an $ 1 $-complemented isometric copy of every Banach space with a monotone basis (see Definition~\ref{univpelcz}). Similar examples have been constructed by J.~Garbuli\'nska-W\c{e}grzyn \cite{garbulinska1, garbulinska2}.

(IV) Theorem~\ref{thmmain} holds for more general classes than the class of non-universal spaces. Let $ Z $ be a separable Banach space for which there are an $ a \in Z $ and a subset $ H \subset Z $ whose closed linear span contains an isometric copy of $ Z $ and such that, for every $ h \in H $, there is an $ \varepsilon > 0 $ with $ \Vert a \pm \varepsilon h \Vert = \Vert a \Vert $. Then the theorem holds for the class of spaces with a monotone basis not containing an isometric copy of $ Z $. Among the universal space $ Z = C(\{ 0, 1 \}^{\mathbb{N}}) $, the required property is fulfilled e.g. by the spaces $ Z = c_{0} $ and $ Z = \ell_{1} $.

(V) If a separable Banach space $ X $ is isomorphically universal for separable Schur spaces, then it is actually isomorphically universal for all separable Banach spaces. This follows from methods in \cite{bossard2} (see \cite[Corollary~51]{braga}). We are able to prove the isometric version of this statement (see Remark~\ref{remschur}).

It is not known if the class of Schur spaces with a basis has the property from Theorem~\ref{thmargdod}. It is not clear whether the tree space method can be used in this case. However, the property is fulfilled by the related class of $ \ell_{1} $-saturated spaces with a basis (see \cite[Theorem~91]{argyrosdodos}).

\section{Preliminaries} \label{sec:prelim}

By $ \Lambda^{< \mathbb{N}} $ we denote the set of all finite sequences of elements of a set $ \Lambda $, including the empty sequence $ \emptyset $. That is,
$$ \Lambda^{< \mathbb{N}} = \bigcup _{l=0}^{\infty } \Lambda^{l} $$
where $ \Lambda^{0} = \{ \emptyset \} $. The length of an $ \eta \in \Lambda^{< \mathbb{N}} $ is denoted by $ |\eta| $. If $ \eta \in \Lambda^{< \mathbb{N}} $ and $ \nu \in \Lambda^{< \mathbb{N}} \cup \Lambda^{\mathbb{N}} $, then by $ \eta \subset \nu $ we mean that $ \eta $ is an initial segment of $ \nu $, i.e., the length of $ \eta $ is less than or equal to the length of $ \nu $ and $ \eta (i) = \nu (i) $ for $ 1 \leq i \leq |\eta| $. By $ (n_{1}, \dots, n_{k})^{\wedge} n $ we mean $ (n_{1}, \dots, n_{k}, n) $. A subset $ T $ of $ \Lambda^{< \mathbb{N}} $ is called a \emph{tree on $ \Lambda $} if
$$ \eta \subset \nu \; \& \; \nu \in T \quad \Rightarrow \quad \eta \in T. $$ 

Moreover, a set $ T \subset \Lambda^{< \mathbb{N}} \setminus \{ \emptyset \} $ is called an \emph{unrooted tree on $ \Lambda $} if $ T \cup \{ \emptyset \} $ is a tree on $ \Lambda $. An (unrooted) tree $ T $ is called \emph{pruned} if every $ \eta \in T $ has a proper extension $ \nu \supset \eta, \nu \neq \eta, \nu \in T $. The set of all infinite branches of $ T $, i.e., sequences $ \nu \in \Lambda^{\mathbb{N}} $ such that $ T $ contains all non-empty initial segments of $ \nu $, is denoted by $ [T] $. An (unrooted) tree $ T $ is called \emph{well-founded} if it does not have an infinite branch.

A \emph{Polish space (topology)} means a separable completely metrizable space (topology). A set $ P $ equipped with a $ \sigma $-algebra is called a \emph{standard Borel space} if the $ \sigma $-algebra is generated by a Polish topology on $ P $. A subset of a standard Borel space is called \emph{analytic} if it is a Borel image of a Polish space.

The following lemma can be found e.g. in~\cite[(25.2)]{kechris}.

\begin{lemma} \label{prunedtree}
A subset $ A \subset \mathbb{N}^{\mathbb{N}} $ is analytic if and only if there is a pruned tree $ T $ on $ \mathbb{N} \times \mathbb{N} $ such that $ A = p[T] $ where $ p : \mathbb{N}^{\mathbb{N}} \times \mathbb{N}^{\mathbb{N}} \to \mathbb{N}^{\mathbb{N}} $ denotes the projection on the first coordinate.
\end{lemma}

For a topological space $ X $, the set $ \mathcal{F}(X) $ of all closed subsets of $ X $ is equipped with the \emph{Effros-Borel structure}, defined as the $ \sigma $-algebra generated by the sets
$$ \{ F \in \mathcal{F}(X) : F \cap U \neq \emptyset \} $$
where $ U $ varies over open subsets of $ X $. If $ X $ is Polish, then, equipped with this $ \sigma $-algebra, $ \mathcal{F}(X) $ forms a standard Borel space.

The \emph{standard Borel space of separable Banach spaces} is defined by
$$ \mathcal{SE}(C([0,1])) = \big\{ F \in \mathcal{F}(C([0,1])) : \textrm{$ F $ is linear} \big\} $$
and considered as a subspace of $ \mathcal{F}(C([0,1])) $.

For a separable Banach space $ X $ and an $ F \subset B_{X^{*}} $, let
$$ F'_{\varepsilon} = F \setminus \bigcup \big\{ U \subset X^{*} : \textrm{$ U $ is $ w^{*} $-open}, \mathrm{diam} (U \cap F) < \varepsilon \big\}, \quad \varepsilon > 0, $$
and recursively
$$ F^{(0)}_{\varepsilon} = F, \quad F^{(\alpha)}_{\varepsilon} = \bigcap_{\beta < \alpha} (F^{(\beta)}_{\varepsilon})'_{\varepsilon}, \quad \varepsilon > 0. $$
We define
$$ \mathrm{Sz}_{\varepsilon}(F) = \min \big( \{ \omega_{1} \} \cup \{ \alpha < \omega_{1} : F^{(\alpha)}_{\varepsilon} = \emptyset \} \big), \quad \varepsilon > 0, $$
$$ \mathrm{Sz}(F) = \sup \{ \mathrm{Sz}_{\varepsilon}(F) : \varepsilon > 0 \}. $$
The \emph{Szlenk index of $ X $} is defined by $ \mathrm{Sz}(X) = \mathrm{Sz}(B_{X^{*}}) $.

For an (unrooted) tree $ T $ and a system $ \{ x_{\eta } : \eta \in T \} $ of elements of a Banach space, we define
$$ \sum _{\eta \in T} x_{\eta } = \lim _{S \to T} \sum _{\eta \in S} x_{\eta } \quad \quad \textrm{(if the limit exists)} $$
where the limit is taken over all finite subtrees $ S \subset T $ directed by inclusion.

The notions and notation we use but do not introduce here are classical and well explained e.g. in \cite{fhhmpz} and \cite{kechris}.

\section{The initial tree space construction}

In this section, we introduce our basic tool for constructing tree spaces. Basically, two ways have been developed how to extract the norm of a tree space from the norms of the subspaces supported by infinite branches (excluding the norm constructed in \cite{kurka}). The first way, based on the well known James tree space \cite{james}, was employed mainly in works of B.~Bossard \cite{bossard2} and S.~A.~Argyros and P.~Dodos \cite{argyrosdodos}.

However, we follow the second way which is more suitable for isometric problems. The method was introduced by B.~Bossard \cite{bossard1} and employed later by G.~Godefroy \cite{godefroy} and G.~Godefroy and N.~J.~Kalton \cite{godkal}. In fact, the tree space from the following definition is a simplified version of the original tree space from \cite{bossard1} which will be introduced later in Definition~\ref{def:B} nevertheless.

\begin{definition} \label{def:E}
Let $ \Lambda $ be a countable set and let $ T $ be a pruned unrooted tree on $ \Lambda $. For every $ \sigma \in [T] $, let $ (F_{\sigma}, \Vert \cdot \Vert_{\sigma}) $ be a Banach space with a monotone basis $ f^{\sigma}_{1}, f^{\sigma}_{2}, \dots $ and let these bases have the property that $ f^{\sigma}_{1}, f^{\sigma}_{2}, \dots, f^{\sigma}_{l} $ and $ f^{\varphi}_{1}, f^{\varphi}_{2}, \dots, f^{\varphi}_{l} $ are $ 1 $-equivalent whenever $ \sigma $ and $ \varphi $ have the same initial segment of length $ l $.

Let us consider the norm on $ c_{00}(T) $ defined by
\begin{equation} \label{E001}
\Vert x \Vert = \sup_{\sigma \in [T]} \Big\Vert \sum_{\eta \subset \sigma} x(\eta) f^{\sigma}_{|\eta|} \Big\Vert_{\sigma}
\end{equation}
and, for every unrooted subtree $ S \subset T $, the projection
\begin{equation} \label{E002}
P_{S}x = \mathbf{1}_{S} \cdot x.
\end{equation}
From the monotonicity of the bases $ f^{\sigma}_{n} $, we obtain
\begin{equation} \label{E003}
\Vert P_{S}x \Vert \leq \Vert x \Vert.
\end{equation}

Finally, we define $ E $ as  a completion of $ (c_{00}(T), \Vert \cdot \Vert) $. The members of the canonical basis of $ c_{00}(T) $ will be denoted by $ e_{\eta} $ (i.e., $ e_{\eta} = \mathbf{1}_{\{ \eta \}} $). We note that the system $ \{ e_{\eta} : \eta \in T \} $ is a basis of $ E $, which follows from the observation that the property $ x = \lim_{S \to T} P_{S}x $ extends from $ c_{00}(T) $ to its closure $ E $, due to the uniform boundedness of the projections $ P_{S} $. The basis is monotone in the sense of formula (\ref{E003}).

Since $ \{ e_{\eta} : \eta \in T \} $ is a basis of $ E $, we are allowed to consider all elements of $ E $ as systems $ x = \{ x(\eta) \}_{\eta \in T} $ of scalars. In this way, formulae \eqref{E001}, \eqref{E002} and \eqref{E003} remain valid for every $ x \in E $. We will denote the members of the corresponding dual system by $ e^{*}_{\eta} $ (i.e., $ e^{*}_{\eta}(x) = x(\eta) $).

For every $ \sigma \in [T] $, we further define spaces
\begin{equation} \label{E004}
\begin{aligned}
E_{\sigma} & = \{ x \in E : \eta \not\subset \sigma \Rightarrow x(\eta) = 0 \}, \\
E^{*}_{\sigma} & = \{ x^{*} \in E^{*} : \eta \not\subset \sigma \Rightarrow x^{*}(e_{\eta}) = 0 \}
\end{aligned}
\end{equation}
and a projection
\begin{equation} \label{E005}
P_{\sigma} = P_{\{ (\sigma_{1}), (\sigma_{1}, \sigma_{2}), \dots \}}.
\end{equation}
We also denote
\begin{equation} \label{E006}
\Phi = \bigcup_{\sigma \in [T]} B_{E_{\sigma}} \quad \textrm{and} \quad \Psi = \bigcup_{\sigma \in [T]} B_{E^{*}_{\sigma}}.
\end{equation}
\end{definition}

\begin{fact} \label{factisigma}
For every $ \sigma \in [T] $, the basis $ f^{\sigma}_{1}, f^{\sigma}_{2}, \dots $ of $ F_{\sigma} $ is $ 1 $-equivalent with the basis $ e_{(\sigma_{1})}, e_{(\sigma_{1}, \sigma_{2})}, \dots $ of $ E_{\sigma} $. In particular, the space $ E $ contains an $ 1 $-complemented isometric copy of $ F_{\sigma} $ for every $ \sigma \in [T] $.
\end{fact}

\begin{proof}
Let $ f = \sum_{n=1}^{\infty} r_{n} f^{\sigma}_{n} $ and $ x = \sum_{n=1}^{\infty} r_{n} e_{(\sigma_{1}, \dots, \sigma_{n})} $ where $ r_{n} \neq 0 $ for finitely many indices $ n $ only. We have
$$ \Big\Vert \sum_{\eta \subset \sigma} x(\eta) f^{\sigma}_{|\eta|} \Big\Vert_{\sigma} = \Big\Vert \sum_{n=1}^{\infty} r_{n} f^{\sigma}_{n} \Big\Vert_{\sigma} = \Vert f \Vert_{\sigma}, $$
and so it remains to check that
$$ \Big\Vert \sum_{\nu \subset \tau} x(\nu) f^{\tau}_{|\nu|} \Big\Vert_{\tau} \leq \Vert f \Vert_{\sigma} $$
for each $ \tau \in [T] \setminus \{ \sigma \} $. Let $ \eta $ be the longest segment such that $ \eta \subset \sigma $ and $ \eta \subset \tau $, and let $ l $ be its length. Then
$$ \Big\Vert \sum_{\nu \subset \tau} x(\nu) f^{\tau}_{|\nu|} \Big\Vert_{\tau} = \Big\Vert \sum_{n=1}^{l} r_{n} f^{\tau}_{n} \Big\Vert_{\tau} = \Big\Vert \sum_{n=1}^{l} r_{n} f^{\sigma}_{n} \Big\Vert_{\sigma} \leq \Vert f \Vert_{\sigma}. $$

The second part of the assertion follows from $ E_{\sigma} = P_{\sigma} E $.
\end{proof}

\begin{fact} \label{factesigma}
For $ x \in E $, we have
$$ \Vert P_{\sigma} x \Vert = \sup_{x^{*} \in B_{E^{*}_{\sigma}}} |x^{*}(x)|. $$
For $ x^{*} \in E^{*} $, we have
$$ \Vert P^{*}_{\sigma} x^{*} \Vert = \sup_{x \in B_{E_{\sigma}}} |x^{*}(x)|. $$
\end{fact}

\begin{proof}
The fact follows directly from the observation that $ P_{\sigma} B_{E} = B_{E_{\sigma}} $ and $ P^{*}_{\sigma} B_{E^{*}} = B_{E^{*}_{\sigma}} $.
\end{proof}

\begin{lemma} \label{psi}
The set $ \Psi $ is compact in the weak$\,^{*}$ topology of $ E^{*} $ and its convex hull is $w^{*}$-dense in $ B_{E^{*}} $.
\end{lemma}

\begin{proof}
To show that $ \Psi $ is $w^{*}$-compact, we just write
$$ \Psi = B_{E^{*}} \setminus \bigcup \Big\{ x^{*} \in E^{*} : x^{*}(e_{\eta}) \neq 0 \, \& \, x^{*}(e_{\nu}) \neq 0 \Big\} $$
where the union is taken over all couples $ \eta, \nu $ of incomparable segments in $ T $. Using (\ref{E001}) in combination with Facts \ref{factisigma} and \ref{factesigma}, we obtain for $ x \in E $ that
$$ \Vert x \Vert = \sup_{\sigma \in [T]} \Vert P_{\sigma} x \Vert = \sup_{\sigma \in [T]} \sup_{x^{*} \in B_{E^{*}_{\sigma}}} |x^{*}(x)| = \sup_{x^{*} \in \Psi} |x^{*}(x)|. $$
Now, to prove that the convex hull of $ \Psi $ is $w^{*}$-dense in $ B_{E^{*}} $, it is sufficient to apply the Hahn-Banach theorem.
\end{proof}

\begin{proposition} \label{shrink}
If the basis $ f^{\sigma}_{1}, f^{\sigma}_{2}, \dots $ is shrinking for every $ \sigma \in [T] $, then the basis $ \{ e_{\eta} : \eta \in T \} $ is also shrinking.
\end{proposition}

\begin{proof}
Let us fix an increasing sequence $ T_{1}, T_{2}, \dots $ of finite unrooted trees with $ \bigcup_{n=1}^{\infty} T_{n} = T $. We show first that
$$ x^{*} \in \Psi \quad \Rightarrow \quad P^{*}_{T_{n}} x^{*} \to x^{*}. $$
Given a $ \sigma \in [T] $, we check the implication for the elements of $ B_{E^{*}_{\sigma}} $. By Fact~\ref{factisigma}, the sequence $ e_{(\sigma_{1})}, e_{(\sigma_{1}, \sigma_{2})}, \dots $ is a shrinking basis of $ E_{\sigma} $. By Fact~\ref{factesigma}, the elements of $ E^{*}_{\sigma} $ satisfy
$$ \Vert x^{*} \Vert = \sup_{x \in B_{E_{\sigma}}} |x^{*}(x)|, \quad x^{*} \in E^{*}_{\sigma}. $$
Hence $ E^{*}_{\sigma} $ is (isometric to) the dual of $ E_{\sigma} $ indeed. The dual sequence $ e^{*}_{(\sigma_{1})}, e^{*}_{(\sigma_{1}, \sigma_{2})}, \dots $ is a basis of $ E^{*}_{\sigma} $. It follows that $ P^{*}_{T_{n}} x^{*} \to x^{*} $ for each $ x^{*} \in E^{*}_{\sigma} $.

Now, let $ y^{*} \in B_{E^{*}} $. By Lemma~\ref{psi}, there exists a probability measure $ \mu $ on $ \Psi $ such that
$$ y^{*} = \int_{\Psi} x^{*} \, d\mu(x^{*}). $$
Therefore,
$$
\begin{aligned}
\lim _{n \to \infty} \Vert y^{*} - P^{*}_{T_{n}} y^{*} \Vert & = \lim _{n \to \infty} \Big\Vert \int_{\Psi} (x^{*} - P^{*}_{T_{n}} x^{*}) \, d\mu(x^{*}) \Big\Vert \\
 & \leq \lim _{n \to \infty} \int_{\Psi} \Vert x^{*} - P^{*}_{T_{n}} x^{*} \Vert \, d\mu(x^{*}) \\
 & = \int_{\Psi} \lim _{n \to \infty} \Vert x^{*} - P^{*}_{T_{n}} x^{*} \Vert \, d\mu(x^{*}) = \int_{\Psi} 0 \, d\mu(x^{*}) = 0.
\end{aligned}
$$
This proves that $ y^{*} $ belongs to the closed linear span of the functionals $ e^{*}_{\eta}, \eta \in T $.
\end{proof}

\begin{lemma} \label{phi}
If the space $ F_{\sigma} $ is reflexive for every $ \sigma \in [T] $, then the set $ \Phi $ is compact in the weak topology of $ E $.
\end{lemma}

\begin{proof}
Let $ x_{1}, x_{2}, \dots $ be a sequence in $ \Phi $. We want to find a subsequence $ x_{n_{k}} $ which converges weakly to an $ x \in \Phi $. By Proposition~\ref{shrink}, it is sufficient to check that
$$ x_{n_{k}}(\eta) \to x(\eta), \quad \eta \in T. $$
Using a diagonal argument, we choose the subsequence $ x_{n_{k}} $ so that
$$ x_{n_{k}}(\eta) \to u(\eta), \quad \eta \in T, $$
for a system of scalars $ u = \{ u(\eta) \}_{\eta \in T} $. It remains to show that this system forms the coordinates of an $ x \in \Phi $.

First, we realize that there is a $ \sigma \in [T] $ such that $ u $ is supported by the branch $ \{ (\sigma_{1}), (\sigma_{1}, \sigma_{2}), \dots \} $. Indeed, if $ u(\eta) \neq 0 \neq u(\nu) $ for some incomparable $ \eta, \nu \in T $, then $ x_{n_{k}}(\eta) \neq 0 \neq x_{n_{k}}(\nu) $ for a large enough $ k $, which is not allowed by the definition of $ \Phi $.

By Fact \ref{factisigma} and our assumption, $ E_{\sigma} $ is reflexive. A subsequence of $ P_{\sigma} x_{n_{k}} $ converges weakly to an $ x \in B_{E_{\sigma}} $, and this limit satisfies $ x(\eta) = u(\eta) $ for every $ \eta \in T $.
\end{proof}

\begin{remark} \label{remschur}
(a) If $ S \subset T $ is a well-founded unrooted subtree, then the subspace
\begin{equation} \label{schurmapping}
H(S) = \overline{\mathrm{span}} \{ e_{\eta}^{*} : \eta \in S \}
\end{equation}
of $ E^{*} $ has the Schur property. Let us assume that $ H(S) $ is not Schur and denote $ H_{\nu}(S) = \overline{\mathrm{span}} \{ e_{\eta}^{*} : \eta \in S \& \nu \subset \eta \} $ for $ \nu \in T \cup \{ \emptyset \} $. It is sufficient to prove that
$$ H_{\nu}(S) \textrm{ is not Schur} \quad \Rightarrow \quad H_{\nu^{\wedge} n}(S) \textrm{ is not Schur for some $ n $}, $$
as this allows us to find an infinite branch of $ S $. One can show that
$$ \Big\Vert \sum_{n=1}^{m} x^{*}_{n} \Big\Vert \geq \frac{1}{2} \sum_{n=1}^{m} \Vert x^{*}_{n} \Vert, \quad x^{*}_{n} \in H_{\nu^{\wedge} n}(S), \, n = 1, 2, \dots, m. $$
Therefore, a hyperplane of $ H_{\nu}(S) $ (or $ H_{\nu}(S) $ itself when $ \nu = \emptyset $) is isomorphic to the $ \ell_{1} $-sum of $ H_{\nu^{\wedge} 1}(S), H_{\nu^{\wedge} 2}(S), \dots $, and the implication follows.

(b) If a separable Banach space $ X $ contains an isometric copy of every separable Schur space, then it contains an isometric copy of every separable Banach space. To show this, we follow a method of B.~Bossard \cite{bossard2}. Let $ x_{1}, x_{2}, \dots $ be a monotone basis of $ C([0, 1]) $ (see e.g. \cite[p. 34]{diestel}) and $ f_{1}, f_{2}, \dots $ be the dual basic sequence in $ C([0, 1])^{*} $. Let $ T = \mathbb{N}^{< \mathbb{N}} \setminus \{ \emptyset \} $, $ f^{\sigma}_{n} = f_{n} $ and $ F_{\sigma} = \overline{\mathrm{span}} \{ f_{n} : n \in \mathbb{N} \} $ for every $ \sigma \in \mathbb{N}^{\mathbb{N}} $. In this setting, let $ H(S) $ be given by \eqref{schurmapping}. Let $ \mathrm{Tr} $ be the subspace of $ 2^{T} $ consisting of all unrooted trees on $ \mathbb{N} $ and let $ \mathrm{WF} $ be the set of all well-founded $ S \in \mathrm{Tr} $. Consider the set
$$ \mathcal{A} = \{ S \in \mathrm{Tr} : \textrm{$ X $ contains an isometric copy of $ H(S) $} \}. $$
Then $ \mathcal{A} $ is analytic (see \cite[Lemma 7 and Lemma 8]{godefroy}) and it contains $ \mathrm{WF} $, due to our assumption. Since $ \mathrm{WF} $ is not analytic (see e.g. \cite[(27.1)]{kechris}), there is an $ S \in \mathcal{A} \setminus \mathrm{WF} $. So, $ X $ contains an isometric copy of $ H(S) $ for some $ S \notin \mathrm{WF} $, which contains an isometric copy of $ C([0, 1]) $.

(c) Assume that, for every $ \sigma \in [T] $, the space $ \overline{\mathrm{span}} \{ f^{*}_{n} : n \in \mathbb{N} \} $ has the Schur property, where $ f^{*}_{1}, f^{*}_{2}, \dots $ is the dual basic sequence of the basis $ f^{\sigma}_{1}, f^{\sigma}_{2}, \dots $. We do not know whether $ H = \overline{\mathrm{span}} \{ e_{\eta}^{*} : \eta \in T \} $ has necessarily the Schur property in such a case. It is possible to show that $ B_{H} = \overline{\mathrm{co}} (H \cap \Psi) $ and that every weakly convergent sequence in $ H \cap \Psi $ is convergent, but this does not seem to be sufficient for $ H $ to be Schur.
\end{remark}

\section{The interpolation method}

The aim of this section is to provide a reflexive variant of the tree space from Definition~\ref{def:E}. As well as the authors of \cite{argyrosdodos}, we apply the Davis-Figiel-Johnson-Pe\l czy\'nski interpolation method.

\begin{definition}[\cite{dfjp}]
Let $ W $ be a bounded, closed, convex and symmetric subset of a Banach space $ X $. For each $ n \in \mathbb{N} $, let $ \Vert \cdot \Vert_{n} $ be the equivalent norm given by
$$ B_{(X, \Vert \cdot \Vert_{n})} = \overline{2^{n}W + 2^{-n}B_{X}}. $$
The \emph{$ 2 $-interpolation space of the pair $ (X, W) $} is defined as the space $ (Y, \opnorm{\cdot}) $ where
$$ \opnorm{x} = \Big( \sum_{n=1}^{\infty} \Vert x \Vert_{n}^{2} \Big)^{1/2}, \quad x \in X, $$
and
$$ Y = \{ x \in X : \opnorm{x} < \infty \}. $$
\end{definition}

\begin{lemma} \label{interpproj}
Let $ P : X \to X $ be a projection such that $ \Vert P \Vert \leq 1 $ and $ PW \subset W $. Then we have
$$ \opnorm{Px} \leq \opnorm{x}, \quad x \in X. $$
If, moreover, $ PW = PB_{X} $, then there is a constant $ c > 0 $ such that
$$ \opnorm{x} = c \Vert x \Vert, \quad x \in PX. $$
In particular, $ Y $ contains an $ 1 $-complemented isometric copy of $ PX $.
\end{lemma}

\begin{proof}
The inequality $ \opnorm{Px} \leq \opnorm{x} $ (which can be proven quite easily actually) follows from \cite[p. 316, Lemma~1(viii)]{dfjp}. To provide a suitable constant $ c > 0 $, it is sufficient to show that
$$ \Vert x \Vert_{n} = \frac{1}{2^{n}+2^{-n}} \Vert x \Vert, \quad x \in PX. $$
Let $ x \in PX $ be given. We will assume that $ \Vert x \Vert = 1 $. Since $ x \in B_{X} $ and $ x = Px \in PB_{X} = PW \subset W $, we have $ (2^{n}+2^{-n}) x \in B_{(X, \Vert \cdot \Vert_{n})} $. Therefore, $ (2^{n}+2^{-n}) \Vert x \Vert_{n} \leq 1 = \Vert x \Vert $.

Let $ 0 < \theta < 1/\Vert x \Vert_{n} $ be chosen arbitrarily. We have $ \Vert \theta x \Vert_{n} < 1 $, and so $ \theta x $ belongs to $ 2^{n}W + 2^{-n}B_{X} $. We obtain $ \theta x = 2^{n} w + 2^{-n} y $ for some $ w \in W $ and $ y \in B_{X} $. Since $ Pw \in PW = PB_{X} \subset B_{X} $, we have $ \theta x = \theta Px = 2^{n} Pw + 2^{-n} Py $ and $ \theta \Vert x \Vert \leq 2^{n} + 2^{-n} $. Due to the choice of $ \theta $, the inequality $ \Vert x \Vert \leq (2^{n} + 2^{-n}) \Vert x \Vert_{n} $ follows.
\end{proof}

\begin{definition}  \label{def:A}
Adopting the notation from Definition~\ref{def:E}, we define $ A $ as the $ 2 $-interpolation space of the pair $ (E, \overline{\mathrm{co}} \, \Phi) $.
\end{definition}

\begin{fact} \label{factA1}
The system $ \{ e_{\eta} : \eta \in T \} $ is a monotone basis of $ A $.
\end{fact}

\begin{proof}
The associated projections $ P_{S} $ satisfy $ \opnorm{P_{S}x} \leq \opnorm{x} $ by \eqref{E003} and Lemma~\ref{interpproj}. The fact thus follows from \cite[p. 316, Lemma~1(ix)]{dfjp}.
\end{proof}

\begin{fact} \label{factA2}
$ A $ contains an $ 1 $-complemented isometric copy of $ F_{\sigma} $ for every $ \sigma \in [T] $.
\end{fact}

\begin{proof}
Recall that $ F_{\sigma} $ is isometric to $ E_{\sigma} = P_{\sigma} E $ by Fact~\ref{factisigma}. The assumptions of Lemma~\ref{interpproj} are met for $ P = P_{\sigma} $, since $ P_{\sigma} (\overline{\mathrm{co}} \, \Phi) \subset \overline{\mathrm{co}} \, \Phi $ and $ P_{\sigma} (\overline{\mathrm{co}} \, \Phi) = B_{E_{\sigma}} = P_{\sigma} B_{E} $.
\end{proof}

\begin{proposition} \label{reflexivity}
If the space $ F_{\sigma} $ is reflexive for every $ \sigma \in [T] $, then the space $ A $ is also reflexive.
\end{proposition}

\begin{proof}
By Lemma~\ref{phi} and the Krein-Smulian theorem, the set $ \overline{\mathrm{co}} \, \Phi $ is weakly compact. To show that $ A $ is reflexive, it is sufficient to apply \cite[p. 313, Lemma~1(iv)]{dfjp}.
\end{proof}

\section{A rotund version of the tree space}

The following definition of a tree space is based on a construction from \cite{bossard1} which was applied also in \cite{godefroy} and \cite{godkal}. Regarding the results from these papers, it is not surprising that this tree space preserves strict convexity of the norm. However, it turns out that the method is suitable also for amalgamating spaces which are not isometrically universal (see Proposition~\ref{nurotB}).

\begin{definition} \label{def:B}
Let $ \Lambda, T, (F_{\sigma}, \Vert \cdot \Vert_{\sigma}) $ and $ f^{\sigma}_{n} $ be as in Definition \ref{def:E}. Suppose moreover that there are positive constants $ c_{1}, c_{2}, \dots $ such that, for every $ \sigma \in [T] $,
\begin{equation} \label{B001}
\Vert \pi_{n}f \Vert_{\sigma}^{2} \geq \Vert \pi_{n-1}f \Vert_{\sigma}^{2} + c_{n}^{2} |f^{*}_{n}(f)|^{2}, \quad f \in F_{\sigma}, \, n \in \mathbb{N},
\end{equation}
where $ f^{*}_{1}, f^{*}_{2}, \dots $ is the dual basic sequence and $ \pi_{0}, \pi_{1}, \dots $ is the sequence of partial sum operators associated with the basis $ f^{\sigma}_{1}, f^{\sigma}_{2}, \dots $.

For every $ x \in c_{00}(T) $, let us consider the formulae
\begin{equation} \label{B002}
\opnorm{x}_{\sigma}^{2} = \Big\Vert \sum_{\eta \subset \sigma} x(\eta) f^{\sigma}_{|\eta|} \Big\Vert_{\sigma}^{2} + \sum_{\eta \not\subset \sigma} c_{|\eta|}^{2} |x(\eta)|^{2}, \quad \sigma \in [T],
\end{equation}
\begin{equation} \label{B003}
\opnorm{x} = \sup_{\sigma \in [T]} \opnorm{x}_{\sigma}
\end{equation}
and
\begin{equation} \label{B004}
P_{S}x = \mathbf{1}_{S} \cdot x,
\end{equation}
where $ S \subset T $ is an unrooted subtree. From the monotonicity of the bases $ f^{\sigma}_{n} $, we obtain
\begin{equation} \label{B005}
\opnorm{P_{S}x} \leq \opnorm{x}.
\end{equation}

Finally, we define $ B $ as  a completion of $ (c_{00}(T), \opnorm{\cdot}) $. Again, the system $ \{ b_{\eta} = \mathbf{1}_{\{ \eta \}} : \eta \in T \} $ is a basis of $ B $ which is monotone in the sense of formula (\ref{B005}). Therefore, we are allowed to consider all elements of $ B $ as systems $ x = \{ x(\eta) \}_{\eta \in T} $ of scalars. In this way, formulae \eqref{B002}, \eqref{B003}, \eqref{B004} and \eqref{B005} remain valid for every $ x \in B $.

For every $ \sigma \in [T] $, we further denote
\begin{equation} \label{B006}
B_{\sigma} = \{ x \in B : \eta \not\subset \sigma \Rightarrow x(\eta) = 0 \},
\end{equation}
\begin{equation} \label{B007}
P_{\sigma} = P_{\{ (\sigma_{1}), (\sigma_{1}, \sigma_{2}), \dots \}}.
\end{equation}
\end{definition}

\begin{fact} \label{factbsigma}
For every $ \sigma \in [T] $, the basis $ f^{\sigma}_{1}, f^{\sigma}_{2}, \dots $ of $ F_{\sigma} $ is $ 1 $-equivalent with the basis $ b_{(\sigma_{1})}, b_{(\sigma_{1}, \sigma_{2})}, \dots $ of $ B_{\sigma} $. In particular, the space $ B $ contains an $ 1 $-complemented isometric copy of $ F_{\sigma} $ for every $ \sigma \in [T] $.
\end{fact}

We do not prove the fact, as an analogous statement appeared in \cite{godefroy} and \cite{godkal}. Actually, the fact can be proven similarly as Fact~\ref{factisigma}, with the difference that (\ref{B001}) is applied. The proof of the following lemma, which is essentially contained in \cite[p. 186]{godkal}, is also skipped.

\begin{lemma}[\cite{godkal}] \label{supattaining}
For every $ x \in B $, the supremum in (\ref{B003}) is attained.
\end{lemma}

\begin{lemma} \label{linesegm}
Let $ [u, v] $ be a non-degenerate line segment in $ B $ such that $ \opnorm{\cdot} $ is constant on $ [u, v] $. Let $ w = \frac{1}{2}(u+v) $ and the supremum in (\ref{B003}) for $ x = w $ be attained at a $ \sigma \in [T] $. Then $ v - u \in B_{\sigma} $ and $ [P_{\sigma}u, P_{\sigma}v] $ is also a non-degenerate line segment on which $ \opnorm{\cdot} $ is constant.
\end{lemma}

\begin{proof}
Let us consider a seminorm
$$ |x|_{\sigma}^{2} = \sum_{\eta \not\subset \sigma} c_{|\eta|}^{2} |x(\eta)|^{2}, \quad x \in B. $$
Using Fact \ref{factbsigma}, we obtain
$$ \opnorm{x}_{\sigma}^{2} = \opnorm{P_{\sigma}x}^{2} + |x|_{\sigma}^{2}, \quad x \in B. $$
We can compute
$$ \opnorm{w} = \opnorm{w}_{\sigma} \leq \frac{1}{2}(\opnorm{u}_{\sigma}+\opnorm{v}_{\sigma}) \leq \frac{1}{2}(\opnorm{u}+\opnorm{v}) = \opnorm{w}, $$
and it is clear that all these norms must be equal. Thus,
$$
\begin{aligned}
0 & = 2\opnorm{u}_{\sigma}^{2} + 2\opnorm{v}_{\sigma}^{2} - 4\opnorm{w}_{\sigma}^{2} \\
 & = 2\opnorm{P_{\sigma}u}^{2} + 2\opnorm{P_{\sigma}v}^{2} - 4\opnorm{P_{\sigma}w}^{2} + 2|u|_{\sigma}^{2} + 2|v|_{\sigma}^{2} - 4|w|_{\sigma}^{2} \\
 & = (\opnorm{P_{\sigma}u} - \opnorm{P_{\sigma}v})^{2} + (\opnorm{P_{\sigma}u} + \opnorm{P_{\sigma}v})^{2} - \opnorm{P_{\sigma}(u+v)}^{2} \\
 & \quad + (|u|_{\sigma} - |v|_{\sigma})^{2} + (|u|_{\sigma} + |v|_{\sigma})^{2} - |u+v|_{\sigma}^{2}.
\end{aligned}
$$
It follows that
\begin{equation} \label{linesegm1}
\opnorm{P_{\sigma}u} = \opnorm{P_{\sigma}v}, \quad \opnorm{P_{\sigma}(u+v)} = \opnorm{P_{\sigma}u} + \opnorm{P_{\sigma}v},
\end{equation}
\begin{equation} \label{linesegm2}
|u|_{\sigma} = |v|_{\sigma}, \quad |u+v|_{\sigma} = |u|_{\sigma} + |v|_{\sigma}.
\end{equation}
By (\ref{linesegm1}), the norm $ \opnorm{\cdot} $ is constant on $ [P_{\sigma}u, P_{\sigma}v] $. By (\ref{linesegm2}), the points $ u $ and $ v $ satisfy $ u(\eta) = v(\eta) $ for every $ \eta \not\subset \sigma $. That is, $ u - P_{\sigma}u = v - P_{\sigma}v $. Therefore, $ v - u = P_{\sigma}v - P_{\sigma}u \in B_{\sigma} $ and the segment $ [P_{\sigma}u, P_{\sigma}v] $ is non-degenerate.
\end{proof}

\begin{proposition} \label{nurotB}
{\rm (a)} If none $ F_{\sigma}, \sigma \in [T], $ is isometrically universal for all separable Banach spaces, then $ B $ is also non-universal.

{\rm (b)} If every $ F_{\sigma}, \sigma \in [T], $ is strictly convex, then $ B $ is strictly convex.
\end{proposition}

\begin{proof}
(a) Assume that $ B $ is isometrically universal for all separable Banach spaces. Let us denote
$$ \Delta = \{ 0, 1 \}^{\mathbb{N}}, \quad \Delta(i) = \{ \gamma \in \Delta : \gamma(1) = i \}, \quad i = 0, 1, $$
$$ Z = C(\Delta), \quad Z(i) = \{ h \in Z : \gamma \notin \Delta(i) \Rightarrow h(\gamma) = 0 \}, \quad i = 0, 1. $$
Considering an isometry $ I : Z \to B $, we denote
$$ x = I (\mathbf{1}_{\Delta(0)}). $$
Lemma~\ref{supattaining} provides us with a $ \sigma \in [T] $ at which the supremum in (\ref{B003}) is attained. We claim that the space $ B_{\sigma} $ (and therefore the space $ F_{\sigma} $ by Fact~\ref{factbsigma}) is universal, showing that $ I $ maps $ Z(1) $ into $ B_{\sigma} $.

Given an $ h \in Z(1) $ with $ \Vert h \Vert \leq 1 $, we observe that $ \Vert \mathbf{1}_{\Delta(0)} \Vert = \Vert \mathbf{1}_{\Delta(0)} \pm h \Vert = 1 $, and so $ \opnorm{x} = \opnorm{x \pm Ih} = 1 $. By Lemma~\ref{linesegm}, we have $ Ih \in B_{\sigma} $.

(b) Assume that $ B $ is not strictly convex. It means that $ \opnorm{\cdot} $ is constant on a non-degenerate line segment $ [u, v] $. Let $ x = \frac{1}{2}(u+v) $ and let the supremum in (\ref{B003}) be attained at a $ \sigma \in [T] $ (which is provided by Lemma~\ref{supattaining}). By Lemma~\ref{linesegm}, the space $ B_{\sigma} $ is not strictly convex. Since $ B_{\sigma} $ and $ F_{\sigma} $ are isometric (see Fact~\ref{factbsigma}), the proof is finished.
\end{proof}

\section{Construction of branches} \label{sec:branches}

In the isomorphic setting, it is possible to construct a tree space in a way such that isomorphic copies of the spaces we want to amalgamate are placed on the infinite branches (as mentioned in the introduction). In the isometric setting, we are not allowed to renorm the spaces, and an additional embedding result is demanded.

We prove that a Banach space $ X $ with a monotone basis can be embedded into another (not much bigger) Banach space $ F $ with a monotone basis $ f_{1}, f_{2}, \dots $ such that the subspaces $ \mathrm{span} \{ f_{1}, \dots, f_{d} \} $ are chosen from a countable family of spaces. To this purpose, we employ the following notion which was useful also in \cite{garbulinska1,garbulinska2,garkub,kubis}.

\begin{definition}
A Banach space $ Z $ is called \emph{rational} if $ Z = \mathbb{R}^{d} $ with a norm such that its unit ball is generated by finitely many points whose all coordinates are rational numbers.

We note that the spaces which have a basis consisting of $ d $ elements will be often identified with $ \mathbb{R}^{d} $ in the obvious way.
\end{definition}

The main goal of this section is to prove the following result. Its proof is based on a construction provided in \cite[Section 4]{kurka} (which was based on a construction from \cite{godkal} in turn) but the present method is considerably simpler.

\begin{proposition} \label{rational}
Let $ X $ be a Banach space and $ e_{1}, e_{2}, \dots $ be a monotone basis of $ X $. Then there exists a Banach space $ F $ with a monotone basis $ f_{1}, f_{2}, \dots $ such that:

{\rm (1)} $ F $ is isomorphic to $ \ell _{2}(X) $.

{\rm (2)} If the basis $ e_{1}, e_{2}, \dots $ is shrinking, then the basis $ f_{1}, f_{2}, \dots $ is shrinking.

{\rm (3)} For every $ d \in \mathbb{N} $, the space $ \mathrm{span} \{ f_{1}, f_{2}, \dots , f_{d} \} $, identified with $ \mathbb{R}^{d} $, is rational.

{\rm (4)} $ F $ contains an $ 1 $-complemented isometric copy of $ X $.
\end{proposition}

\begin{definition} \label{def:pi}
By $ \pi $ we will denote the bijection $ \mathbb{N} \to \mathbb{N}^{2} $ given by $ \pi(1) = (1,1), \pi(2) = (1,2), \pi(3) = (2,1), \pi(4) = (1,3) $ etc.
$$
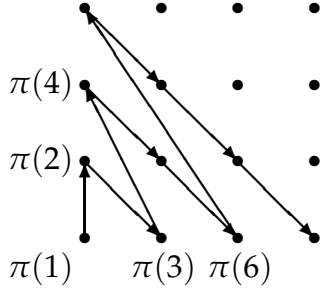
\unitlength=5.8pt
\begin{picture}(20,20)(-3,-3)
\multiput(0,0)(5,0){4}{\circle*{0.7}}
\multiput(0,5)(5,0){4}{\circle*{0.7}}
\multiput(0,10)(5,0){4}{\circle*{0.7}}
\multiput(0,15)(5,0){4}{\circle*{0.7}}
\thicklines
\put(0,0){\vector(0,1){5}}
\put(0,5){\vector(1,-1){5}}
\put(5,0){\vector(-1,2){5}}
\put(0,10){\vector(1,-1){5}}
\put(5,5){\vector(1,-1){5}}
\put(10,0){\vector(-2,3){10}}
\put(0,15){\vector(1,-1){5}}
\put(5,10){\vector(1,-1){5}}
\put(10,5){\vector(1,-1){5}}
\put(-5,-2.5){$ \pi(1) $}
\put(-5,4.5){$ \pi(2) $}
\put(3,-2.5){$ \pi(3) $}
\put(-5,9.5){$ \pi(4) $}
\put(8,-2.5){$ \pi(6) $}
\end{picture}
$$
\end{definition}

\begin{definition} \label{ratconstr}
For every $ d \in \mathbb{N} $, let us fix an ordering of all monotone rational norms on $ \mathbb{R}^{d} $ into a sequence $ |\cdot|_{d,1} $, $ |\cdot|_{d,2} $, $ \dots $ .

Let $ e_{1}, e_{2}, \dots $ be a normalized monotone basis of a Banach space $ (X, \Vert \cdot \Vert_{X}) $. Let $ f_{i} = e_{\pi(i)} $ where $ e_{(n,k)} $ stands for the element of $ \ell _{2}(X) $ which has $ e_{k} $ on the $ n $-th place and $ 0 $ elsewhere. Let us moreover denote
\begin{equation} \label{rat01}
F_{d} = \mathrm{span} \{ f_{1}, f_{2}, \dots , f_{d} \}.
\end{equation}
For every $ d \in \mathbb{N} $, let $ l_{d} = l_{d}(e_{1}, e_{2}, \dots ) $ be the least natural number such that the monotone rational norm $ |\cdot|_{d} = |\cdot|_{d,l_{d}} $ satisfies
\begin{equation} \label{rat02}
\Big( 1 - \frac{1}{2^{2d+1}} \Big) \Vert f \Vert _{\ell _{2}(X)} \leq |f|_{d} \leq \Big( 1 - \frac{1}{2^{2d+2}} \Big) \Vert f \Vert _{\ell _{2}(X)}, \quad f \in F_{d}.
\end{equation}
The formula is valid for some monotone rational norm on $ F_{d} $, as the sequence $ f_{1}, f_{2}, \dots $ is a monotone basis of $ \ell _{2}(X) $.

We define a space $ F = F(e_{1}, e_{2}, \dots ) $ with a norm $ \Vert \cdot \Vert $ by
\begin{equation} \label{rat03}
B_{(F, \Vert \cdot \Vert)} = \overline{\mathrm{co}} \, \bigcup_{d=1}^{\infty} B_{(F_{d}, |\cdot|_{d})}.
\end{equation}
We also define operators
\begin{equation} \label{rat04}
T : \ell _{2}(X) \to X, \quad (x_{1}, x_{2}, \dots ) \mapsto \frac{\sqrt{3}}{2} \Big( x_{1} + \frac{1}{2}x_{2} + \frac{1}{4}x_{3} + \dots \Big),
\end{equation}
\begin{equation} \label{rat05}
U : X \to \ell _{2}(X), \quad x \mapsto \frac{\sqrt{3}}{2} \Big( x, \frac{1}{2}x, \frac{1}{4}x, \dots \Big).
\end{equation}

The sequence of partial sum operators associated with the basis $ f_{1}, f_{2}, \dots $ will be denoted by $ P_{1}, P_{2}, \dots $ .
\end{definition}

\begin{lemma} \label{ratlemma}
We have $ F = \ell _{2}(X) $ and the norm $ \Vert \cdot \Vert $ fulfills
\begin{equation} \label{rat06}
\frac{7}{8} \Vert f \Vert _{\ell _{2}(X)} \leq \Vert f \Vert \leq \Vert f \Vert _{\ell _{2}(X)}, \quad f \in F.
\end{equation}
The basis $ f_{1}, f_{2}, \dots $ forms a monotone basis of $ (F, \Vert \cdot \Vert) $ which is shrinking if the basis $ e_{1}, e_{2}, \dots $ is shrinking. Finally, for every $ d \in \mathbb{N} $, we have
\begin{equation} \label{rat07}
B_{(F_{d}, \Vert \cdot \Vert)} = \mathrm{co} \, \bigcup_{j=1}^{d} B_{(F_{j}, |\cdot|_{j})}.
\end{equation}
In particular, the space $ (F_{d}, \Vert \cdot \Vert) $ is rational.
\end{lemma}

\begin{proof}
By (\ref{rat02}), we have $ \frac{7}{8} \Vert f \Vert _{\ell _{2}(X)} \leq |f|_{d} \leq \Vert f \Vert _{\ell _{2}(X)} $ for $ f \in F_{d} $. Thus,
$$ B_{(F_{d}, |\cdot|_{d})} \subset \frac{8}{7} B_{\ell _{2}(X)} \quad \textrm{and} \quad B_{\ell _{2}(X)} \cap F_{d} \subset B_{(F_{d}, |\cdot|_{d})} \subset B_{(F, \Vert \cdot \Vert)}, $$
and it follows that
$$ B_{(F, \Vert \cdot \Vert)} \subset \frac{8}{7} B_{\ell _{2}(X)} \quad \textrm{and} \quad B_{\ell _{2}(X)} \subset B_{(F, \Vert \cdot \Vert)}. $$
Clearly, if the basis $ e_{1}, e_{2}, \dots $ is shrinking, then the basis $ f_{1}, f_{2}, \dots $ is shrinking. To show that it is monotone with respect to $ \Vert \cdot \Vert $, it is sufficient to realize that the associated partial sum operators $ P_{1}, P_{2}, \dots $ map the unit ball of $ (F_{d}, |\cdot|_{d}) $ into itself, and consequently that the unit ball of $ (F, \Vert \cdot \Vert) $ has the same property. To show (\ref{rat07}), it is sufficient to prove that $ P_{d} $ maps the unit ball of $ (F_{j}, |\cdot|_{j}) $, where $ j > d $, into the unit ball of $ (F_{d}, |\cdot|_{d}) $. For $ f \in F_{j} $, we can compute
$$ |P_{d}f|_{d} \leq \Big( 1 - \frac{1}{2^{2d+2}} \Big) \Vert P_{d}f \Vert _{\ell _{2}(X)} \leq \Big( 1 - \frac{1}{2^{2j+1}} \Big) \Vert f \Vert _{\ell _{2}(X)} \leq |f|_{j}, $$
which completes the proof.
\end{proof}

\begin{lemma} \label{operT}
We have $ \Vert Tf \Vert _{X} \leq \Vert f \Vert $ for $ f \in F $.
\end{lemma}

\begin{proof}
For $ n \in \mathbb{N} $ and $ x_{1}, x_{2}, \dots , x_{n} \in X $, we can write
\begin{align*}
\big\Vert T(x_{1}, & \, x_{2}, \dots, x_{n}, 0, 0, \dots ) \big\Vert _{X} = \frac{\sqrt{3}}{2} \Big\Vert x_{1} + \frac{1}{2}x_{2} + \dots + \frac{1}{2^{n-1}}x_{n} \Big\Vert _{X} \\
 & \leq \frac{\sqrt{3}}{2} \Big( \Vert x_{1} \Vert _{X} + \frac{1}{2} \Vert x_{2} \Vert _{X} + \dots + \frac{1}{2^{n-1}} \Vert x_{n} \Vert _{X} \Big) \\
 & \leq \frac{\sqrt{3}}{2} \sqrt{1 + \frac{1}{4} + \dots + \frac{1}{4^{n-1}}} \sqrt{\Vert x_{1} \Vert _{X}^{2} + \Vert x_{2} \Vert _{X}^{2} + \dots + \Vert x_{n} \Vert _{X}^{2} } \\
 & = \sqrt{1 - \frac{1}{4^{n}}} \, \big\Vert (x_{1}, x_{2}, \dots , x_{n}, 0, 0, \dots ) \big\Vert _{\ell _{2}(X)}.
\end{align*}
It follows that
$$ \Vert Tf \Vert _{X} \leq \sqrt{1 - \frac{1}{4^{d}}} \, \Vert f \Vert _{\ell _{2}(X)}, \quad f \in F_{d}, \; d \in \mathbb{N}, $$
as the elements of $ F_{d} $ are supported by the first $ d $ coordinates (obviously from the definition of $ \pi $).

Now, given $ d \in \mathbb{N} $, we obtain for $ f \in F_{d} $ that
$$ \Vert Tf \Vert _{X} \leq \sqrt{1 - \frac{1}{4^{d}}} \, \Vert f \Vert _{\ell _{2}(X)} \leq \Big( 1 - \frac{1}{2^{2d+1}} \Big) \Vert f \Vert _{\ell _{2}(X)} \leq |f|_{d}. $$
Therefore, the unit ball of $ (F_{d}, |\cdot|_{d}) $, where $ d \in \mathbb{N} $, and consequently the unit ball of $ (F, \Vert \cdot \Vert) $, are subsets of $ \{ f \in F : \Vert Tf \Vert _{X} \leq 1 \} $.
\end{proof}

\begin{lemma} \label{operU}
We have $ \Vert Ux \Vert = \Vert x \Vert _{X} $ for $ x \in X $ and the range of $ U $ is $ 1 $-complemented in $ (F, \Vert \cdot \Vert) $.
\end{lemma}

\begin{proof}
It can be easily shown that
$$ TUx = x \quad \textrm{and} \quad \Vert Ux \Vert _{\ell _{2}(X)} = \Vert x \Vert _{X} $$
for $ x \in X $. Using Lemma \ref{operT} and Lemma \ref{ratlemma}, we can write
$$ \Vert x \Vert _{X} = \Vert TUx \Vert_{X} \leq \Vert Ux \Vert \leq \Vert Ux \Vert_{\ell _{2}(X)} = \Vert x \Vert _{X}, \quad x \in X. $$
Moreover, $ UT : F \to F $ is a projection onto $ UX $ with $ \Vert UT \Vert \leq 1 $.
\end{proof}

The proof of Proposition~\ref{rational} is completed. Nevertheless, we prove here one more lemma which will be useful later.

\begin{lemma} \label{furthlemma}
We have
$$ \Vert f \Vert \geq \Vert P_{n}f \Vert + \frac{1}{2^{2n+4}} \Vert f - P_{n}f \Vert, \quad f \in F, \; n \in \mathbb{N}. $$ 
\end{lemma}

\begin{proof}
Similarly as in the proof of Lemma~\ref{operT}, it is sufficient to show that the unit ball of $ (F_{d}, |\cdot|_{d}) $, where $ d \in \mathbb{N} $, is a subset of $ \{ f \in F : \Vert P_{n}f \Vert + \frac{1}{2^{2n+4}} \Vert f - P_{n}f \Vert \leq 1 \} $. So, we just need to check that
$$ \Vert P_{n}f \Vert + \frac{1}{2^{2n+4}} \Vert f - P_{n}f \Vert \leq |f|_{d}, \quad f \in F_{d}. $$
The inequality is clear when $ d \leq n $, as $ P_{n}f = f $. If $ d \geq n+1 $, then
$$ \Vert P_{n}f \Vert \leq |P_{n}f|_{n} \leq \Big( 1 - \frac{1}{2^{2n+2}} \Big) \Vert P_{n}f \Vert _{\ell _{2}(X)} \leq \Big( 1 - \frac{1}{2^{2n+2}} \Big) \Vert f \Vert _{\ell _{2}(X)}, $$
and so
\begin{align*}
\Vert & P_{n}f \Vert + \frac{1}{2^{2n+4}} \Vert f - P_{n}f \Vert \\
 & = \Big( 1 - \frac{1}{2^{2d+1}} \Big) \Vert P_{n}f \Vert + \frac{1}{2^{2d+1}} \Vert P_{n}f \Vert + \frac{1}{2^{2n+4}} \Vert f - P_{n}f \Vert \\
 & \leq \Big( 1 - \frac{1}{2^{2d+1}} \Big) \Big( 1 - \frac{1}{2^{2n+2}} \Big) \Vert f \Vert _{\ell _{2}(X)} + \frac{1}{2^{2n+3}} \Vert f \Vert + \frac{1}{2^{2n+4}} \cdot 2 \Vert f \Vert \\
 & \leq \Big( 1 - \frac{1}{2^{2n+2}} \Big) |f|_{d} + \frac{1}{2^{2n+3}} |f|_{d} + \frac{1}{2^{2n+4}} \cdot 2 |f|_{d} \\ 
 & = |f|_{d}
\end{align*}
for every $ f \in F_{d} $.
\end{proof}

\section{Renorming I} \label{sec:renormingI}

For the class of reflexive spaces and the class of spaces with a shrinking basis, the construction of the space $ F $ from Definition~\ref{ratconstr} is satisfactory. For the other two classes from Theorem~\ref{thmmain}, the space $ F $ needs to be renormed in a way such that the relevant isometric properties of the initial space $ X $ are preserved.

In fact, we renorm the space in two steps (renormings $ \Vert \cdot \Vert_{I} $ and $ \Vert \cdot \Vert_{II} $). For the class of non-universal spaces, one renorming is sufficient. For the class of strictly convex spaces, one more renorming is needed.

Let us accentuate two aspects of the renormings. Firstly, the new norm on the subspace $ F_{d} = \mathrm{span} \{ f_{1}, \dots, f_{d} \} $ depends only on the old norm on $ F_{d} $ itself. In this way, only countably many possibilities for the norm of $ F_{d} $ may occur. Secondly, the norm is not changed on the subspace $ UX $ which is still an $ 1 $-complemented copy of $ X $.

\begin{definition}
We define a seminorm
\begin{equation} \label{renI01}
\beta(f)^{2} = \sum_{n=1}^{\infty} \sum_{k=1}^{\infty} \frac{1}{2^{4\pi^{-1}(n+1,k)}} |e^{*}_{(n,k)}(f)-2e^{*}_{(n+1,k)}(f)|^{2}, \quad f \in F,
\end{equation}
where $ e^{*}_{(n,k)} $ is the system biorthogonal with the basic system $ e_{(n,k)} $.
\end{definition}

The proof of the following observation is skipped.

\begin{fact} \label{betanull}
For an $ f \in F $, the following assertions are equivalent:

{\rm (i)} $ \beta(f) = 0 $,

{\rm (ii)} $ e^{*}_{(n,k)}(f)-2e^{*}_{(n+1,k)}(f) = 0 $ for all $ n, k $,

{\rm (iii)} $ f \in UX $.
\end{fact}

\begin{lemma} \label{betabound}
Let $ d \in \mathbb{N} \cup \{ 0 \} $. Then every $ f \in \overline{\mathrm{span}} \, \{ f_{d+1}, f_{d+2}, \dots \} $ satisfies
$$ \beta(f) \leq \frac{2}{2^{2d}} \Vert f \Vert. $$
In particular, $ \beta(f) \leq 2 \Vert f \Vert $ for every $ f \in F $.
\end{lemma}

\begin{proof}
We can compute
\begin{equation} \label{renI02}
|e^{*}_{(n,k)}(f)| \leq \Vert e^{*}_{(n,k)} \Vert_{\ell _{2}(X)} \Vert f \Vert_{\ell _{2}(X)} \leq 2 \cdot \frac{8}{7} \Vert f \Vert,
\end{equation}
\begin{equation} \label{renI03}
|e^{*}_{(n,k)}(f)-2e^{*}_{(n+1,k)}(f)| \leq 3 \cdot 2 \cdot \frac{8}{7} \Vert f \Vert \leq 2 \sqrt{15} \Vert f \Vert.
\end{equation}
Moreover, we obtain from $ f \in \overline{\mathrm{span}} \, \{ f_{d+1}, f_{d+2}, \dots \} $ that
$$ \pi^{-1}(n,k) \leq d \quad \Rightarrow \quad e^{*}_{(n,k)}(f) = 0, $$
and consequently
$$ \pi^{-1}(n+1,k) \leq d \quad \Rightarrow \quad e^{*}_{(n,k)}(f)-2e^{*}_{(n+1,k)}(f) = 0. $$
Therefore,
\begin{align*}
\beta(f)^{2} & \leq \sum_{\pi^{-1}(n+1,k) > d} \frac{1}{2^{4\pi^{-1}(n+1,k)}} \cdot \big( 2 \sqrt{15} \Vert f \Vert \big)^{2} \\
 & \leq \sum_{j > d} \frac{1}{2^{4j}} \cdot 4 \cdot 15 \Vert f \Vert^{2} = \frac{4}{2^{4d}} \Vert f \Vert^{2},
\end{align*}
which proves the lemma.
\end{proof}

\begin{definition}
We define
\begin{equation} \label{renI04}
\Vert f \Vert_{I}^{2} = \Vert f \Vert^{2} + \frac{1}{2^{7}} \beta(f)^{2}, \quad f \in F.
\end{equation}
We note that a simple application of Lemma~\ref{betabound} gives
\begin{equation} \label{renI05}
\Vert f \Vert \leq \Vert f \Vert_{I} \leq 2 \Vert f \Vert.
\end{equation}
\end{definition}

\begin{lemma} \label{operUI}
We have $ \Vert Ux \Vert_{I} = \Vert x \Vert _{X} $ for $ x \in X $ and the range of $ U $ is $ 1 $-complemented in $ (F, \Vert \cdot \Vert_{I}) $.
\end{lemma}

\begin{proof}
Using Fact~\ref{betanull} and Lemma~\ref{operU}, we can write $ \Vert Ux \Vert_{I} = \Vert Ux \Vert = \Vert x \Vert _{X} $ for $ x \in X $. The projection $ UT $ works as well as in the proof of Lemma~\ref{operU}, because $ \Vert UTf \Vert_{I} = \Vert UTf \Vert \leq \Vert f \Vert \leq \Vert f \Vert_{I} $ for $ f \in F $.
\end{proof}

\begin{lemma} \label{linesegmI}
Let $ [u, v] $ be a non-degenerate line segment in $ F $ such that $ \Vert \cdot \Vert_{I} $ is constant on $ [u, v] $. Then $ v - u \in UX $.
\end{lemma}

\begin{proof}
By the same argument as in the proof of Lemma~\ref{linesegm}, we arrive at
$$ \Vert u \Vert = \Vert v \Vert, \quad \Vert u+v \Vert = \Vert u \Vert + \Vert v \Vert, $$
$$ \beta(u) = \beta(v), \quad \beta(u+v) = \beta(u) + \beta(v), $$
and consequently
$$ e^{*}_{(n,k)}(u)-2e^{*}_{(n+1,k)}(u) = e^{*}_{(n,k)}(v)-2e^{*}_{(n+1,k)}(v), \quad n, k \in \mathbb{N}. $$
It follows that $ v - u \in UX $ by Fact~\ref{betanull}.
\end{proof}

\begin{proposition} \label{nonuniv}
If $ X $ is not isometrically universal for all separable Banach spaces, then $ (F, \Vert \cdot \Vert_{I}) $ is also non-universal.
\end{proposition}

\begin{proof}
Assume that $ (F, \Vert \cdot \Vert_{I}) $ is isometrically universal for all separable Banach spaces. Again, let us denote
$$ \Delta = \{ 0, 1 \}^{\mathbb{N}}, \quad \Delta(i) = \{ \gamma \in \Delta : \gamma(1) = i \}, \quad i = 0, 1, $$
$$ Z = C(\Delta), \quad Z(i) = \{ h \in Z : \gamma \notin \Delta(i) \Rightarrow h(\gamma) = 0 \}, \quad i = 0, 1. $$
Considering an isometry $ I : Z \to F $, we denote
$$ f = I (\mathbf{1}_{\Delta(0)}). $$
We claim that the space $ UX $ (and thus the space $ X $ by Lemma~\ref{operUI}) is universal, showing that $ I $ maps $ Z(1) $ into $ UX $.

Given an $ h \in Z(1) $ with $ \Vert h \Vert \leq 1 $, we observe that $ \Vert \mathbf{1}_{\Delta(0)} \Vert = \Vert \mathbf{1}_{\Delta(0)} \pm h \Vert = 1 $, and so $ \Vert f \Vert_{I} = \Vert f \pm Ih \Vert_{I} = 1 $. By Lemma~\ref{linesegmI}, we have $ Ih \in UX $.
\end{proof}

\begin{lemma} \label{furthlemmaI}
We have
$$ \Vert f \Vert_{I} \geq \Vert P_{d}f \Vert_{I} + \frac{1}{2^{2d+7}} \Vert f - P_{d}f \Vert_{I}, \quad f \in F, \; d \in \mathbb{N}. $$ 
\end{lemma}

\begin{proof}
By Lemma~\ref{furthlemma},
\begin{align*}
\Vert f \Vert^{2} - \Vert P_{d}f \Vert^{2} & = (\Vert f \Vert + \Vert P_{d}f \Vert) (\Vert f \Vert - \Vert P_{d}f \Vert) \\
 & \geq (\Vert f \Vert + \Vert P_{d}f \Vert) \cdot \frac{1}{2^{2d+4}} \Vert f - P_{d}f \Vert.
\end{align*}
At the same time, by Lemma~\ref{betabound},
\begin{align*}
\beta(P_{d}f)^{2} - \beta(f)^{2} & = (\beta(P_{d}f) + \beta(f)) (\beta(P_{d}f) - \beta(f)) \\
 & \leq (\beta(P_{d}f) + \beta(f)) \cdot \beta(f - P_{d}f) \\
 & \leq 2 (\Vert P_{d}f \Vert + \Vert f \Vert) \cdot \frac{2}{2^{2d}} \Vert f - P_{d}f \Vert. 
\end{align*}
Thus, using \eqref{renI05}, we can compute
\begin{align*}
\Vert f \Vert_{I}^{2} - \Vert P_{d}f \Vert_{I}^{2} & = \Vert f \Vert^{2} - \Vert P_{d}f \Vert^{2} + \frac{1}{2^{7}} \big( \beta(f)^{2} - \beta(P_{d}f)^{2} \big) \\
 & \geq \Big( \frac{1}{2^{2d+4}} - \frac{1}{2^{7}} \cdot \frac{4}{2^{2d}} \Big) (\Vert f \Vert + \Vert P_{d}f \Vert) \cdot \Vert f - P_{d}f \Vert \\
 & = \frac{1}{2^{2d+5}} (\Vert f \Vert + \Vert P_{d}f \Vert) \cdot \Vert f - P_{d}f \Vert \\
 & \geq \frac{1}{2^{2d+5}} \cdot \frac{1}{2} (\Vert f \Vert_{I} + \Vert P_{d}f \Vert_{I}) \cdot \frac{1}{2} \Vert f - P_{d}f \Vert_{I}.
\end{align*}
Now, it is sufficient to divide both sides by $ \Vert f \Vert_{I} + \Vert P_{d}f \Vert_{I} $.
\end{proof}

\section{Renorming II} \label{sec:renormingII}

\begin{definition}
We define a seminorm
\begin{equation} \label{renII01}
\alpha(f)^{2} = \sum_{n=1}^{\infty} \sum_{k=1}^{\infty} \frac{1}{2^{4\pi^{-1}(n,k)}} |e^{*}_{(n,k)}(f)|^{2}, \quad f \in F,
\end{equation}
where $ e^{*}_{(n,k)} $ is the system biorthogonal with the basic system $ e_{(n,k)} $.
\end{definition}

\begin{lemma} \label{alphabound}
We have
$$ \alpha(f) < \Vert f \Vert, \quad 0 \neq f \in F. $$
\end{lemma}

\begin{proof}
Using \eqref{renI02}, we can compute
$$ \alpha(f)^{2} \leq \sum_{n=1}^{\infty} \sum_{k=1}^{\infty} \frac{1}{2^{4\pi^{-1}(n,k)}} \Big( 2 \cdot \frac{8}{7} \Vert f \Vert \Big)^{2} = \frac{1}{15} \Big( 2 \cdot \frac{8}{7} \Vert f \Vert \Big)^{2} < \Vert f \Vert^{2}, $$
which proves the lemma.
\end{proof}

\begin{fact} \label{rhonorm}
There is a norm $ \varrho $ on $ \mathbb{R}^{3} $ such that
\begin{itemize}
\item $ \frac{1}{2}(|r|+|s|) \leq \varrho(r,s,t) \leq \max \{ |r|, |s|, |t| \} $ and, in particular, the unit sphere contains the line segment $ [(1,1,-1),(1,1,1)] $,
\item $ \varrho(r',s',t') \geq \varrho(r,s,t) $ for $ 0 \leq r \leq r', 0 \leq s \leq s', 0 \leq t \leq t' $,
\item $ \varrho(r,s,t') > \varrho(r,s,t) $ for $ 0 < r < s, 0 < t < t' $,
\item $ \varrho(r',s,t) \geq \varrho(r,s,t) + \frac{1}{4}(r'-r) $ for $ r,r',s,t > 0 $, $ 0 < r < r' $.
\end{itemize}
\end{fact}

\begin{proof}[Proof (sketch).]
Let a norm $ \varrho_{0} $ be given by
$$ B_{(\mathbb{R}^{3}, \varrho_{0})} = \mathrm{co} \, \Big( \{ (\pm 1, \pm 1, \pm 1) \} \cup \sqrt{2} B \Big), $$
where $ B $ stands for the Euclidean unit ball of $ \mathbb{R}^{3} $. This norm satisfies the first three properties, and the norm
$$ \varrho(r,s,t) = \frac{1}{4} (|r|+|s|) + \frac{1}{2} \varrho_{0}(r,s,t) $$
satisfies additionally the fourth one.
\end{proof}

\begin{definition}
We define
\begin{equation} \label{renII02}
\Vert f \Vert_{II} = \varrho\big( \Vert f \Vert, \Vert f \Vert_{I}, \alpha(f) \big), \quad f \in F.
\end{equation}
We notice that a simple application of \eqref{renI05} and Lemma~\ref{alphabound} gives
\begin{equation} \label{renII03}
\Vert f \Vert \leq \Vert f \Vert_{II} \leq 2 \Vert f \Vert,
\end{equation}
since
\begin{align*}
\Vert f \Vert & \leq \frac{1}{2}(\Vert f \Vert + \Vert f \Vert_{I}) \leq \varrho\big( \Vert f \Vert, \Vert f \Vert_{I}, \alpha(f) \big) \\
 & \leq \max \big\{ \Vert f \Vert, \Vert f \Vert_{I}, \alpha(f) \big\} = \Vert f \Vert_{I} \leq 2 \Vert f \Vert.
\end{align*}
\end{definition}

\begin{lemma} \label{operUII}
We have $ \Vert Ux \Vert_{II} = \Vert x \Vert _{X} $ for $ x \in X $ and the range of $ U $ is $ 1 $-complemented in $ (F, \Vert \cdot \Vert_{II}) $.
\end{lemma}

\begin{proof}
Let $ x \in X $ be such that $ \Vert x \Vert_{X} = 1 $. By Lemma~\ref{alphabound}, Lemma~\ref{operU} and Lemma~\ref{operUI}, we have $ \alpha(Ux) < \Vert Ux \Vert = 1 = \Vert Ux \Vert_{I} $. Since the unit sphere $ S_{(\mathbb{R}^{3}, \varrho)} $ contains the line segment $ [(1,1,-1),(1,1,1)] $, we obtain $ \Vert Ux \Vert_{II} = 1 = \Vert x \Vert_{X} $. The projection $ UT $ still works, because $ \Vert UTf \Vert_{II} = \Vert Tf \Vert_{X} = \Vert UTf \Vert \leq \Vert f \Vert \leq \Vert f \Vert_{II} $ for $ f \in F $.
\end{proof}

\begin{lemma} \label{linesegmII}
Let $ [u, v] $ be a non-degenerate line segment in $ F $ such that $ \Vert \cdot \Vert_{II} $ is constant on $ [u, v] $. Then $ u $ and $ v $ belong to $ UX $.
\end{lemma}

\begin{proof}
It is enough to show that $ w = \frac{1}{2}(u+v) \in UX $ (the argument can be repeated for any subsegment of $ [u, v] $). Assume the opposite, i.e., $ w \notin UX $. We have $ \beta(w) > 0 $ by Fact~\ref{betanull}, and so $ \Vert w \Vert < \Vert w \Vert_{I} $. Using the inequality
$$ \alpha(w) < \frac{1}{2}\big( \alpha(u) + \alpha(v) \big), $$
a property of $ \varrho $ provides
$$ \varrho\Big( \Vert w \Vert, \Vert w \Vert_{I}, \frac{1}{2}\big( \alpha(u) + \alpha(v) \big) \Big) > \varrho \big( \Vert w \Vert, \Vert w \Vert_{I}, \alpha(w) \big) = \Vert w \Vert_{II}. $$
The computation
\begin{align*}
\frac{1}{2}\big( \Vert u \Vert_{II} + & \Vert v \Vert_{II} \big) = \frac{1}{2} \Big( \varrho \big( \Vert u \Vert, \Vert u \Vert_{I}, \alpha(u) \big) + \varrho \big( \Vert v \Vert, \Vert v \Vert_{I}, \alpha(v) \big) \Big) \\
 & \geq \varrho \Big( \frac{1}{2} \big( \Vert u \Vert + \Vert v \Vert \big), \frac{1}{2} \big( \Vert u \Vert_{I} + \Vert v \Vert_{I} \big), \frac{1}{2} \big( \alpha(u) + \alpha(v) \big) \Big) \\
 & \geq \varrho\Big( \Vert w \Vert, \Vert w \Vert_{I}, \frac{1}{2}\big( \alpha(u) + \alpha(v) \big) \Big) > \Vert w \Vert_{II} 
\end{align*}
concludes the proof.
\end{proof}

\begin{proposition} \label{strictconv}
If $ X $ is strictly convex, then $ (F, \Vert \cdot \Vert_{II}) $ is also strictly convex.
\end{proposition}

\begin{proof}
This follows from Lemma~\ref{operUII} and Lemma~\ref{linesegmII}.
\end{proof}

\begin{lemma} \label{furthlemmaII}
We have
$$ \Vert f \Vert_{II} \geq \Vert P_{d}f \Vert_{II} + \frac{1}{2^{2d+7}} \Vert f - P_{d}f \Vert_{II}, \quad f \in F, \; d \in \mathbb{N}. $$ 
\end{lemma}

\begin{proof}
Using Lemma~\ref{furthlemma} and Lemma~\ref{furthlemmaI}, we can compute
\begin{align*}
\Vert & f \Vert_{II} = \varrho\big( \Vert f \Vert, \Vert f \Vert_{I}, \alpha(f) \big) \\
 & \geq \varrho\Big( \Vert P_{d}f \Vert + \frac{1}{2^{2d+4}} \Vert f - P_{d}f \Vert, \Vert P_{d}f \Vert_{I} + \frac{1}{2^{2d+7}} \Vert f - P_{d}f \Vert_{I}, \alpha(f) \Big) \\
 & \geq \varrho\Big( \Vert P_{d}f \Vert + \frac{1}{2^{2d+4}} \Vert f - P_{d}f \Vert, \Vert P_{d}f \Vert_{I}, \alpha(P_{d}f) \Big) \\
 & \geq \varrho\big( \Vert P_{d}f \Vert, \Vert P_{d}f \Vert_{I}, \alpha(P_{d}f) \big) + \frac{1}{4} \cdot \frac{1}{2^{2d+4}} \Vert f - P_{d}f \Vert \\
 & \geq \Vert P_{d}f \Vert_{II} + \frac{1}{4} \cdot \frac{1}{2^{2d+4}} \cdot \frac{1}{2} \Vert f - P_{d}f \Vert_{II},
\end{align*}
which proves the lemma.
\end{proof}

\section{Amalgamations of Asplund and reflexive spaces}

In the final stage of the proof of Theorem~\ref{thmmain}, we need some further notation. We introduce a coding of all rational Banach spaces whose basis is monotone. This enables us to provide a version of the Pe\l czy\'nski universal space.

\begin{definition}
We fix a system $ \{(Z_{\eta}, \Vert \cdot \Vert_{\eta})\}_{\eta \in \mathbb{N}^{< \mathbb{N}}} $ of rational Banach spaces which satisfies the following requirements.

(a) For every $ \eta $, the basis of $ Z_{\eta} $, denoted by $ z^{\eta}_{1}, z^{\eta}_{2}, \dots, z^{\eta}_{|\eta|} $, is monotone and consists of $ |\eta| $ members in a way such that, for any two comparable sequences $ \eta \subset \nu $, the space $ Z_{\nu} $ is an extension of $ Z_{\eta} $ in the sense that the basis $ z^{\eta}_{1}, z^{\eta}_{2}, \dots, z^{\eta}_{|\eta|} $ is $ 1 $-equivalent with $ z^{\nu}_{1}, z^{\nu}_{2}, \dots, z^{\nu}_{|\eta|} $.

(b) Every monotone rational extension of $ Z_{\eta} $ is included as $ Z_{\nu} $ for some $ \nu \supset \eta $. More precisely, if $ Z $ is a rational space whose basis $ z_{1}, z_{2}, \dots, z_{d} $ is monotone and such that $ z^{\eta}_{1}, z^{\eta}_{2}, \dots, z^{\eta}_{|\eta|} $ is $ 1 $-equivalent with $ z_{1}, z_{2}, \dots, z_{|\eta|} $, then there is a $ \nu \supset \eta $ with $ |\nu| = d $ such that $ z^{\nu}_{1}, z^{\nu}_{2}, \dots, z^{\nu}_{|\nu|} $ is $ 1 $-equivalent with $ z_{1}, z_{2}, \dots, z_{d} $.
\end{definition}

\begin{definition}
For every $ \varphi \in \mathbb{N}^{\mathbb{N}} $, let $ (Z_{\varphi}, \Vert \cdot \Vert_{\varphi}) $ be a Banach space with a monotone basis $ z^{\varphi}_{1}, z^{\varphi}_{2}, \dots $ such that, for every $ \eta \subset \varphi $, the basis $ z^{\eta}_{1}, z^{\eta}_{2}, \dots, z^{\eta}_{|\eta|} $ of $ Z_{\eta} $ is $ 1 $-equivalent with $ z^{\varphi}_{1}, z^{\varphi}_{2}, \dots, z^{\varphi}_{|\eta|} $.
\end{definition}

\begin{definition} \label{univpelcz}
Let $ U $ be a completion of $ c_{00}(\mathbb{N}^{< \mathbb{N}} \setminus \{ \emptyset \} ) $ with the norm defined by one of the equivalent formulae
\begin{equation}
\Vert x \Vert = \sup_{\nu \in \mathbb{N}^{< \mathbb{N}}} \Big\Vert \sum_{\eta \subset \nu} x(\eta) z^{\nu}_{|\eta|} \Big\Vert_{\nu},
\end{equation}
\begin{equation}
\Vert x \Vert = \sup_{\varphi \in \mathbb{N}^{\mathbb{N}}} \Big\Vert \sum_{\eta \subset \varphi} x(\eta) z^{\varphi}_{|\eta|} \Big\Vert_{\varphi}.
\end{equation}
Further, let $ \varpi : \mathbb{N} \to \mathbb{N}^{< \mathbb{N}} \setminus \{ \emptyset \} $ be a fixed non-decreasing bijection and let
\begin{equation}
u_{i} = \mathbf{1}_{\{ \varpi(i) \}}, \quad i \in \mathbb{N}.
\end{equation}

As the space $ U $ is defined according to Definition~\ref{def:E}, several remarkable properties follow. First of all, the sequence $ u_{1}, u_{2}, \dots $ is a monotone basis of $ U $. If we denote
\begin{equation} \label{U001}
\Delta : \varphi \in \mathbb{N}^{\mathbb{N}} \mapsto \{ \varpi^{-1}((\varphi_{1})) < \varpi^{-1}((\varphi_{1}, \varphi_{2})) < \dots \} \subset \mathbb{N},
\end{equation}
then, using Fact~\ref{factisigma}, the sequences $ \{ z^{\varphi}_{n} : n \in \mathbb{N} \} $ and $ \{ u_{i} : i \in \Delta(\varphi) \} $ are $ 1 $-equivalent for every $ \varphi \in \mathbb{N}^{\mathbb{N}} $. The copy $ \overline{\mathrm{span}} \{ u_{i} : i \in \Delta(\varphi) \} $ of $ Z_{\varphi} $ is $ 1 $-complemented in $ U $. Moreover, due to Proposition~\ref{rational}, every Banach space $ X $ with a monotone basis has an $ 1 $-complemented isometric copy in $ Z_{\varphi} $ for some $ \varphi \in \mathbb{N}^{\mathbb{N}} $. It follows that $ X $ has an $ 1 $-complemented isometric copy also in $ U $.

We note that the space $ U $, including its construction and properties, is fairly similar to the space constructed and studied in \cite{garbulinska1}.
\end{definition}

\begin{lemma} \label{lemmaszlenki}
Let $ \mathcal{C} $ be an analytic set of Banach spaces with separable dual. Then there is a $ \beta < \omega_{1} $ such that $ \mathrm{Sz}(\ell_{2}(X)) \leq \beta $ for every $ X \in \mathcal{C} $.
\end{lemma}

\begin{proof}
It follows from \cite[Theorem~4.11]{bossard2} and \cite[Proposition~0.1(ii)]{bossard2} that $ \sup \{ \mathrm{Sz}(X) : X \in \mathcal{C}' \} < \omega_{1} $ for any analytic set $ \mathcal{C}' $ of Banach spaces with separable dual. So, it is sufficient to find an analytic set $ \mathcal{C}' $ which contains an isomorphic copy of $ \ell_{2}(X) $ for every $ X \in \mathcal{C} $ and every $ Y \in \mathcal{C}' $ is isomorphic to $ \ell_{2}(X) $ for some $ X \in \mathcal{C} $.

Let us consider an isometry $ I : \ell_{2}(C([0, 1])) \to C([0, 1]) $ and let $ \kappa : C([0, 1]) \to C([0, 1]) $ be defined by $ \kappa(X) = I(\ell_{2}(X)) $ where $ \ell_{2}(X) $ is considered as a subspace of $ \ell_{2}(C([0, 1])) $. As $ \kappa $ is a Borel mapping, $ \mathcal{C}' = \kappa(\mathcal{C}) $ works.
\end{proof}

\begin{lemma} \label{lemmaszlenkii}
For every $ \beta < \omega_{1} $, the set
\begin{equation} \label{lemmaszlenkii1}
\mathcal{A} = \big\{ \varphi \in \mathbb{N}^{\mathbb{N}} : \textrm{$ \mathrm{Sz}(Z_{\varphi}) \leq \beta $ and $ z^{\varphi}_{1}, z^{\varphi}_{2}, \dots $ is shrinking} \big\}
\end{equation}
is Borel in $ \mathbb{N}^{\mathbb{N}} $.
\end{lemma}

\begin{proof}
By \cite[Theorem~5.4(i)]{bossard2} and \cite[Proposition~0.1(i)]{bossard2}, the set
\begin{align*}
\mathcal{B} = \Big\{ \{ i_{1} < i_{2} < \dots \} \subset \mathbb{N} : \mathrm{Sz} & (\overline{\mathrm{span}} \{ u_{i_{1}}, u_{i_{2}}, \dots \} ) \leq \beta \\
 & \textrm{and $ u_{i_{1}}, u_{i_{2}}, \dots $ is shrinking} \Big\}
\end{align*}
is Borel in the space of all subsets of $ \mathbb{N} $. As $ \Delta $ is a continuous mapping, it remains to realize that $ \mathcal{A} = \Delta^{-1}(\mathcal{B}) $.
\end{proof}

\begin{proof}[Proof of Theorem \ref{thmmain}, Part 1. (Shrinking basis case.)]
Let $ \mathcal{C} $ be an analytic set of Banach spaces such that every member admits a monotone shrinking basis. Let $ \beta < \omega_{1} $ be as in Lemma~\ref{lemmaszlenki} and let $ \mathcal{A} $ be given by \eqref{lemmaszlenkii1}. By Lemma~\ref{lemmaszlenkii}, $ \mathcal{A} $ is Borel, and thus analytic. Notice that Proposition~\ref{rational} guarantees that every $ X \in \mathcal{C} $ has an $ 1 $-complemented isometric copy in $ Z_{\varphi} $ for some $ \varphi \in \mathcal{A} $.

By Lemma~\ref{prunedtree}, there is an unrooted pruned tree $ T $ on $ \mathbb{N} \times \mathbb{N} $ such that $ \mathcal{A} = p[T] $ where $ p $ denotes the projection on the first coordinate. Let us consider the collection
$$ (F_{\sigma}, \Vert \cdot \Vert_{\sigma}) = (Z_{p(\sigma)}, \Vert \cdot \Vert_{p(\sigma)}), \quad f^{\sigma}_{n} = z^{p(\sigma)}_{n}, \quad \sigma \in [T], \, n \in \mathbb{N}. $$
In this way, the collection $ F_{\sigma}, \sigma \in [T], $ consists of the same spaces as the collection $ Z_{\varphi}, \varphi \in \mathcal{A} $.

Finally, let $ E $ be the space constructed in Definition~\ref{def:E} for this collection. This space admits the required properties, due to Fact~\ref{factisigma} and Proposition~\ref{shrink}.
\end{proof}

\begin{lemma} \label{lemmarefl}
For an analytic set $ \mathcal{C} $ of Banach spaces, the set
\begin{equation} \label{lemmarefl1}
\mathcal{A} = \big\{ \varphi \in \mathbb{N}^{\mathbb{N}} : \textrm{$ Z_{\varphi} $ is isomorphic to $ \ell_{2}(X) $ for some $ X \in \mathcal{C} $} \big\}
\end{equation}
is analytic in $ \mathbb{N}^{\mathbb{N}} $.
\end{lemma}

\begin{proof}
It is easy to show (see the proof of Lemma~\ref{lemmaszlenki}) that there is an analytic set $ \mathcal{C}' $ which contains an isomorphic copy of $ \ell_{2}(X) $ for every $ X \in \mathcal{C} $ and every $ Y \in \mathcal{C}' $ is isomorphic to $ \ell_{2}(X) $ for some $ X \in \mathcal{C} $. By \cite[Theorem 2.3(i)]{bossard2}, the saturation
\begin{align*}
\mathcal{C}'' & = \big\{ Z \in \mathcal{SE}(C([0,1])) : \textrm{$ Z $ is isomorphic to some $ Y \in \mathcal{C}' $} \big\} \\
& = \big\{ Z \in \mathcal{SE}(C([0,1])) : \textrm{$ Z $ is isomorphic to $ \ell_{2}(X) $ for an $ X \in \mathcal{C} $} \big\}
\end{align*}
is analytic.

Let $ I : U \to C([0, 1]) $ be an isometry. It is easy to show that the mapping
$$ \zeta : \mathbb{N}^{\mathbb{N}} \to \mathcal{SE}(C([0,1])), \quad \varphi \mapsto \overline{\mathrm{span}} \{ I(\mathbf{1}_{\{ (\varphi_{1}) \}}), I(\mathbf{1}_{\{ (\varphi_{1}, \varphi_{2}) \}}), \dots \}, $$
is Borel. Due to Fact~\ref{factisigma}, the spaces $ Z_{\varphi} $ and $ \zeta(\varphi) $ are isometric. It follows that $ \mathcal{A} = \zeta^{-1}(\mathcal{C}'') $, and so that $ \mathcal{A} $ is analytic.
\end{proof}

\begin{proof}[Proof of Theorem \ref{thmmain}, Part 2. (Reflexive case.)]
Let $ \mathcal{C} $ be an analytic set of reflexive Banach spaces such that every member has a monotone basis. Let $ \mathcal{A} $ be given by \eqref{lemmarefl1}. By Lemma~\ref{lemmarefl}, $ \mathcal{A} $ is analytic. Notice that Proposition~\ref{rational} guarantees that every $ X \in \mathcal{C} $ has an $ 1 $-complemented isometric copy in $ Z_{\varphi} $ for some $ \varphi \in \mathcal{A} $. At the same time, the space $ Z_{\varphi} $ is reflexive for every $ \varphi \in \mathcal{A} $.

By Lemma~\ref{prunedtree}, there is an unrooted pruned tree $ T $ on $ \mathbb{N} \times \mathbb{N} $ such that $ \mathcal{A} = p[T] $ where $ p $ denotes the projection on the first coordinate. Let us consider the collection
$$ (F_{\sigma}, \Vert \cdot \Vert_{\sigma}) = (Z_{p(\sigma)}, \Vert \cdot \Vert_{p(\sigma)}), \quad f^{\sigma}_{n} = z^{p(\sigma)}_{n}, \quad \sigma \in [T], \, n \in \mathbb{N}. $$
In this way, the collection $ F_{\sigma}, \sigma \in [T], $ consists of the same spaces as the collection $ Z_{\varphi}, \varphi \in \mathcal{A} $.

Finally, let $ A $ be the space established in Definition~\ref{def:A} for this collection. This space admits the required properties, due to Facts~\ref{factA1}, \ref{factA2} and Proposition~\ref{reflexivity}.
\end{proof}

\section{Amalgamations of non-universal and rotund spaces}

\begin{definition}
Let $ \varphi \in \mathbb{N}^{\mathbb{N}} $ and let $ z^{*}_{1}, z^{*}_{2}, \dots $ denote the dual basic sequence of $ z^{\varphi}_{1}, z^{\varphi}_{2}, \dots $. Let us define seminorms
\begin{equation} \label{nurot01}
\alpha(z)^{2} = \sum_{i=1}^{\infty} \frac{1}{2^{4i}} |z^{*}_{i}(z)|^{2} \quad \Big( = \sum_{n=1}^{\infty} \sum_{k=1}^{\infty} \frac{1}{2^{4\pi^{-1}(n,k)}} |z^{*}_{\pi^{-1}(n,k)}(z)|^{2} \Big),
\end{equation}
\begin{equation} \label{nurot02}
\beta(z)^{2} = \sum_{n=1}^{\infty} \sum_{k=1}^{\infty} \frac{1}{2^{4\pi^{-1}(n+1,k)}} |z^{*}_{\pi^{-1}(n,k)}(z)-2z^{*}_{\pi^{-1}(n+1,k)}(z)|^{2},
\end{equation}
where $ \pi $ is introduced in Definition~\ref{def:pi}. Let us further define
\begin{equation} \label{nurot03}
Z_{\varphi}^{I} = \{ z \in Z_{\varphi} : \beta(z) < \infty \}, \quad Z_{\varphi}^{II} = \{ z \in Z_{\varphi}^{I} : \alpha(z) < \infty \},
\end{equation}
\begin{equation} \label{nurot04}
\Vert z \Vert_{\varphi, I}^{2} = \Vert z \Vert_{\varphi}^{2} + \frac{1}{2^{7}} \beta(z)^{2}, \quad z \in Z_{\varphi}^{I},
\end{equation}
\begin{equation} \label{nurot05}
\Vert z \Vert_{\varphi, II} = \varrho\big( \Vert z \Vert_{\varphi}, \Vert z \Vert_{\varphi, I}, \alpha(z) \big), \quad z \in Z_{\varphi}^{II},
\end{equation}
where $ \varrho $ is a norm given by Fact~\ref{rhonorm}.
\end{definition}

\begin{definition}
The subspace of $ (S_{C([0, 1])})^{\mathbb{N}} $ consisting of all normalized monotone basic sequences will be denoted by $ \mathcal{M} $.
\end{definition}

The following proposition summarizes most of the results from Sections~\ref{sec:branches}, \ref{sec:renormingI} and \ref{sec:renormingII}.

\begin{proposition} \label{summ}
There exists a Borel mapping $ \Theta : \mathcal{M} \to \mathbb{N}^{\mathbb{N}} $ such that, for every $ (e_{1}, e_{2}, \dots) \in \mathcal{M} $, if we denote $ \varphi = \Theta(e_{1}, e_{2}, \dots) $ and $ X = \overline{\mathrm{span}} \{ e_{1}, e_{2}, \dots \} $, then:

{\rm (1)} $ Z_{\varphi}^{I} = Z_{\varphi}^{II} = Z_{\varphi} $ and the norms fulfill
$$ \Vert z \Vert_{\varphi} \leq \Vert z \Vert_{\varphi, I} \leq 2 \Vert z \Vert_{\varphi}, \quad \Vert z \Vert_{\varphi} \leq \Vert z \Vert_{\varphi, II} \leq 2 \Vert z \Vert_{\varphi}, \quad z \in Z_{\varphi}. $$

{\rm (2)} Both $ Z_{\varphi}^{I} $ and $ Z_{\varphi}^{II} $ contain an $ 1 $-complemented isometric copy of $ X $.

{\rm (3)} If $ X $ is not isometrically universal for all separable Banach spaces, then $ Z_{\varphi}^{I} $ is also non-universal.

{\rm (4)} If $ X $ is strictly convex, then $ Z_{\varphi}^{II} $ is also strictly convex.

{\rm (5)} We have
$$ \Vert P_{n}z \Vert_{\varphi, I}^{2} \geq \Vert P_{n-1}z \Vert_{\varphi, I}^{2} + \Big( \frac{7}{2^{2n+8}} \Big)^{2} |z^{*}_{n}(z)|^{2}, \quad z \in Z_{\varphi}^{I}, \, n \in \mathbb{N}, $$
$$ \Vert P_{n}z \Vert_{\varphi, II}^{2} \geq \Vert P_{n-1}z \Vert_{\varphi, II}^{2} + \Big( \frac{7}{2^{2n+8}} \Big)^{2} |z^{*}_{n}(z)|^{2}, \quad z \in Z_{\varphi}^{II}, \, n \in \mathbb{N}, $$
where $ z^{*}_{1}, z^{*}_{2}, \dots $ is the dual basic sequence and $ P_{0}, P_{1}, \dots $ is the sequence of partial sum operators associated with the basis $ z^{\varphi}_{1}, z^{\varphi}_{2}, \dots $ .
\end{proposition}

\begin{proof}
We realize first that the functions $ l_{d} : \mathcal{M} \to \mathbb{N} $ from Definition~\ref{ratconstr} are Borel. If $ l \in \mathbb{N} $, then the set of basic sequences $ e_{1}, e_{2}, \dots $ for which \eqref{rat02} holds with $ |\cdot|_{d} = |\cdot|_{d,l} $ forms a closed set. Therefore, the set of sequences with $ l_{d} = l $ is the difference of two closed sets.

Now, if a monotone basic sequence $ e_{1}, e_{2}, \dots $ is given, the properties of the system $ \{(Z_{\varphi}, \Vert \cdot \Vert_{\varphi})\}_{\varphi \in \mathbb{N}^{\mathbb{N}}} $ together with Lemma~\ref{ratlemma} guarantee that there is a $ \varphi = (\varphi_{1}, \varphi_{2}, \dots) $ such that $ (Z_{\varphi}, \Vert \cdot \Vert_{\varphi}) $ and $ (F, \Vert \cdot \Vert) $ coincide, including their bases. To show that the choice of $ \varphi $ can be Borel, it is sufficient to realize that $ \varphi $ can be constructed recursively in the way that $ \varphi_{d} $ depends only on $ l_{1}, \dots, l_{d} $. This is allowed by formula \eqref{rat07} which implies that the norm on $ \mathrm{span} \{ f_{1}, f_{2}, \dots , f_{d} \} $ is determined by the values $ l_{1}, \dots, l_{d} $.

Let us check the required properties. Notice that the spaces $ Z_{\varphi}^{I} $ and $ Z_{\varphi}^{II} $ coincide with $ (F, \Vert \cdot \Vert_{I}) $ and $ (F, \Vert \cdot \Vert_{II}) $. So, the properties easily follow from lemmata and propositions proven above.

Property (1) follows from \eqref{renI05} and \eqref{renII03} and property (2) follows from Lemma~\ref{operUI} and Lemma~\ref{operUII}. Property (3) follows from Proposition~\ref{nonuniv} and property (4) follows from Proposition~\ref{strictconv}. Finally, property (5) needs a little calculation. By Lemma~\ref{ratlemma}, we have
$$ \frac{7}{8} |z^{*}_{n}(z)| \leq |z^{*}_{n}(z)| \Vert z^{\varphi}_{n} \Vert_{\varphi} = \Vert P_{n}z - P_{n-1}z \Vert_{\varphi} \leq \Vert P_{n}z - P_{n-1}z \Vert_{\varphi, I}. $$
Using Lemma~\ref{furthlemmaI}, we obtain
\begin{align*}
\Vert P_{n}z \Vert_{\varphi, I}^{2} & \geq \Big( \Vert P_{n-1}z \Vert_{\varphi, I} + \frac{1}{2^{2(n-1)+7}} \Vert P_{n}z - P_{n-1}z \Vert_{\varphi, I} \Big)^{2} \\
 & \geq \Big( \Vert P_{n-1}z \Vert_{\varphi, I} + \frac{7}{2^{2n+8}} |z^{*}_{n}(z)| \Big)^{2} \\
 & \geq \Vert P_{n-1}z \Vert_{\varphi, I}^{2} + \Big( \frac{7}{2^{2n+8}} \Big)^{2} |z^{*}_{n}(z)|^{2}.
\end{align*}
The proof of the analogous inequality for $ Z_{\varphi}^{II} $ is the same, we just use Lemma~\ref{furthlemmaII} instead of Lemma~\ref{furthlemmaI}.
\end{proof}

\begin{proof}[Proof of Theorem \ref{thmmain}, Part 3.]
Suppose that $ \mathcal{C} $ is an analytic set of Banach spaces such that every member has a monotone basis. Let $ \mathcal{M}_{\mathcal{C}} $ be the subset of $ \mathcal{M} $ consisting of all normalized monotone bases of members of $ \mathcal{C} $. Since the mapping
$$ (e_{1}, e_{2}, \dots) \mapsto \overline{\mathrm{span}} \{ e_{1}, e_{2}, \dots \} $$
is Borel, the pre-image $ \mathcal{M}_{\mathcal{C}} $ of $ \mathcal{C} $ is analytic. Let $ \Theta $ be the mapping from Proposition~\ref{summ}. The image $ \Theta(\mathcal{M}_{\mathcal{C}}) $ is an analytic subset of $ \mathbb{N}^{\mathbb{N}} $. By Lemma~\ref{prunedtree}, there is an unrooted pruned tree $ T $ on $ \mathbb{N} \times \mathbb{N} $ such that $ \Theta(\mathcal{M}_{\mathcal{C}}) = p[T] $ where $ p $ denotes the projection on the first coordinate. Let us consider the collections
$$ (F_{\sigma}^{I}, \Vert \cdot \Vert_{\sigma, I}) = (Z_{p(\sigma)}^{I}, \Vert \cdot \Vert_{p(\sigma), I}), \quad (F_{\sigma}^{II}, \Vert \cdot \Vert_{\sigma, II}) = (Z_{p(\sigma)}^{II}, \Vert \cdot \Vert_{p(\sigma), II}), $$
$$ f^{\sigma}_{n} = z^{p(\sigma)}_{n}, \quad \sigma \in [T], \, n \in \mathbb{N}. $$
Finally, let $ B^{I} $ and $ B^{II} $ be the spaces constructed in Definition~\ref{def:B} for these collections. Note that property (5) from Proposition~\ref{summ} guarantees that the requirement \eqref{B001} is fulfilled.

Both spaces $ B^{I} $ and $ B^{II} $ contain an $ 1 $-complemented isometric copy of every $ X \in \mathcal{C} $. Indeed, a monotone basis of $ X $ is contained in $ \mathcal{M}_{\mathcal{C}} $, and so property (2) from Proposition~\ref{summ} is satisfied for some $ \varphi \in \Theta(\mathcal{M}_{\mathcal{C}}) = p[T] $. If we choose a $ \sigma \in [T] $ with $ p(\sigma) = \varphi $, then $ X $ has an $ 1 $-complemented isometric copy in $ F_{\sigma}^{I} $ and in $ F_{\sigma}^{II} $, and it is sufficient to apply Fact~\ref{factbsigma}.
 
If every $ X \in \mathcal{C} $ is non-universal (strictly convex), then $ B^{I} $ is non-universal ($ B^{II} $ is strictly convex). Indeed, in such a case, property (3) (property (4)) from Proposition~\ref{summ} implies that the spaces $ Z_{\varphi}^{I}, \varphi \in \Theta(\mathcal{M}_{\mathcal{C}}) $ ($ Z_{\varphi}^{II}, \varphi \in \Theta(\mathcal{M}_{\mathcal{C}}) $), and so the spaces $ F_{\sigma}^{I}, \sigma \in [T] $ ($ F_{\sigma}^{II}, \sigma \in [T] $), are non-universal (strictly convex), and it remains just to apply Proposition~\ref{nurotB}.
\end{proof}

\end{document}